\documentclass[letterpaper,10pt]{amsart}

\usepackage[all]{xy}                        %

\CompileMatrices                            % Faster

\UseTips                                    % Use

\input xypic
\usepackage[bookmarks=true]{hyperref}       % Hyperref
\usepackage{amssymb,latexsym,amsmath,amscd}
\usepackage{xspace}

\usepackage{graphicx}
\reversemarginpar

\theoremstyle{plain}
\newtheorem{theorem}{Theorem}[section]
\newtheorem*{theorem*}{Theorem}
\newtheorem{proposition}[theorem]{Proposition}
\newtheorem{corollary}[theorem]{Corollary}
\newtheorem{lemma}[theorem]{Lemma}

\theoremstyle{definition}
\newtheorem{definition}[theorem]{Definition}

\newtheorem{remark}[theorem]{Remark}

\newcommand{\enm}[1]{\ensuremath{#1}}          %
\newcommand{\op}[1]{\operatorname{#1}}
\newcommand{\cal}[1]{\mathcal{#1}}

\newcommand{\CC}{\enm{\mathbb{C}}}

\newcommand{\ZZ}{\enm{\mathbb{Z}}}

\newcommand{\PP}{\enm{\mathbb{P}}}

\newcommand{\Bb}{\enm{\cal{B}}}
\newcommand{\Aa}{\enm{\cal{A}}}
\newcommand{\Cc}{\enm{\cal{C}}}
\newcommand{\Dd}{\enm{\cal{D}}}
\newcommand{\Ee}{\enm{\cal{E}}}
\newcommand{\Ff}{\enm{\cal{F}}}
\newcommand{\Gg}{\enm{\cal{G}}}

\newcommand{\Ii}{\enm{\cal{I}}}

\newcommand{\Ll}{\enm{\cal{L}}}
\newcommand{\Mm}{\enm{\cal{M}}}
\newcommand{\Nn}{\enm{\cal{N}}}
\newcommand{\Oo}{\enm{\cal{O}}}

\newcommand{\Rr}{\enm{\cal{R}}}

\newcommand{\Uu}{\enm{\cal{U}}}
\newcommand{\Vv}{\enm{\cal{V}}}
\newcommand{\Ww}{\enm{\cal{W}}}

\renewcommand{\phi}{\varphi}
\renewcommand{\theta}{\vartheta}
\renewcommand{\epsilon}{\varepsilon}

\newcommand{\Ext}{\op{Ext}}

\renewcommand{\to}[1][]{\xrightarrow{\ #1\ }}

\newcommand{\old}[1]{}

\begin{document}

%\layout
\title[Globally generated vector bundles]{Globally generated vector bundles \\on a smooth quadric surface}
\author{E. Ballico, S. Huh and F. Malaspina}
\address{Universit\`a di Trento, 38123 Povo (TN), Italy}
\email{edoardo.ballico@unitn.it}
\address{Department of Mathematics, Sungkyunkwan University, Suwon 440-746, Korea}
\email{sukmoonh@skku.edu}
\address{Politecnico di Torino, Corso Duca degli Abruzzi 24, 10129 Torino, Italy}
\email{francesco.malaspina@polito.it}
\thanks{The second author is supported by Basic Science Research Program 2010-0009195 through NRF funded by MEST}

\begin{abstract}
We give the classification of globally generated vector bundles of rank $2$ on a smooth quadric surface with $c_1\le (2,2)$ in terms of the indices of the bundles, and extend the result to arbitrary higher rank case. We also investigate their indecomposability and give the sufficient and necessary condition on numeric data of vector bundles for indecomposability.
\end{abstract}

\maketitle
%\layout

%\tableofcontents

\section{Introduction}
Globally generated vector bundles on projective varieties play an important role in projective algebraic geometry and their classification has been done quite recently over projective spaces by many people \cite{ACM}\cite{am}\cite{ce}\cite{e}\cite{m}\cite{SU}\cite{SU2}. Especially in \cite{e}, Ellia determines the Chern classes of globally generated vector bundles of rank $2$ on $\PP^2$, using several results on L\"uroth semigroup of smooth plane curves \cite{c}\cite{gr}. In \cite{ce} Chiodera and Ellia classify the globally generated vector bundles of rank $2$ on $\PP^n$ with small first Chern classes. We ask similar questions over a smooth quadric surface $Q$ and give answers as in our previous works \cite{BHM1}\cite{BHM2}. Our main theorem is the following:

\begin{theorem}
There exists an indecomposable and globally generated vector bundle of rank $r\ge 2$ on $Q$ with the Chern classes $(c_1, c_2)$ such that $c_1=(a,b)\leq (2,2)$, $a\le b$ if and only if $(c_1,c_2;r)$ is one of the following:
\begin{align*}
&(1,1,2;r=2,3),\\
&(1,2,2;2),(1,2,3;r=2,3),(1,2,4;r=2,3,4,5),\\
&(2,2,3;2),(2,2,4;r=2,3), (2,2,5;r=2,3),(2,2,6;r=2,3,4,5),\\
&(2,2,8;r=2,3,4,5,6,7,8)
\end{align*}
\end{theorem}

In the second section we fix the notations and we explain the preliminaries. In the third section we show that any globally generated vector bundle with $c_1=(0,0)$ or $c_1=(a,0)$ 
and either $a\le 2$ or rank two splits. In proposition \ref{1.1} and proposition \ref{prop1.2} we classify indecomposable rank two globally generated vector bundles with $c_1=(1,1)$ and $c_1=(1,2)$. In the fourth section we deal with the case $c_1=(2,2)$. In the lemmas \ref{nost}, \ref{le1}, \ref{le2} and \ref{le3} we give a complete classification. We can find indecomposable rank two globally generated vector bundles for any $3\leq c_2\leq 8$ except $c_2=7$.

We apply an old method of associating to a rank $2$ vector bundle $\Ff$ on $Q$ a $0$-dimensional subscheme $Z$ on $Q$, and relate the global generatedness of $\Ff$ and the ideal sheaf $\Ii_Z$ twisted by $\Oo_Q(c_1)$ with $c_1=c_1(\Ff)$. Since $\Ii_Z(c_1)$ is globally generated, $Z$ is contained in a complete intersection of two hypersurfaces in $|\Oo_Q(c_1)|$. It gives us an upper bound for the possible second Chern class of $\Ff$. We then check the global generatedness case by case.

In the last section we study the higher rank case. The goal of this section is to describe the possible rank and Chern classes with which indecomposable vector bundles on $Q$ exist (see theorem \ref{hr}). If $\Ee$ is a globally generated vector bundle on $Q$ of rank $r\ge 3$,  then $\Ee$ can be expressed as an extension of a globally generated vector bundle $\Ff$ of rank $2$ by trivial factors. In other words, $\Ee$ can be determined by linearly independent elements in $\Ext ^1( \Ff, \Oo_Q)$ and it gives us an upper bound for the possible rank of indecomposable vector bundle. Then we check the existence of indecomposable and globally generated vector bundles case by case with these bounds.

We are grateful to the anonymous referees for several crucial corrections and suggestions.

%%%%%%%%%%%%%%%%%%%%%%%%%%%%%%%%%
\section{Preliminaries}
Throughout the article, our base field is $\CC$, the field of complex numbers.

 Let $Q$ be a smooth quadric surface isomorphic to $\PP V_1 \times \PP V_2$ for two 2-dimensional vector spaces $V_1$ and $V_2$ and then it is embedded into $\PP^3\simeq \PP V$ by the Segre map, where $V=V_1\otimes V_2$. Let us denote $f^*\Oo_{\PP^1}(a) \otimes g^*\Oo_{\PP^1}(b)$ by $\Oo_Q(a,b)$ and $\Ee\otimes \Oo_Q(a,b)$ by $\Ee(a,b)$ for coherent sheaves $\Ee$ on $Q$, where $f$ and $g$ are the projections from $Q$ to each factors. Then the canonical line bundle $\omega_Q$ of $Q$ is $\Oo_Q(-2,-2)$. We also denote by $\Ee^\vee$ the dual sheaf of $\Ee$.

For a coherent sheaf $\Ee$ of rank $r$ on $Q$ with the Chern classes $c_1=(a,b)\in \ZZ^{\oplus 2}$ and $c_2=c\in \ZZ$, we have :
\begin{align*}
c_1(\Ee (s,t))&=(a+rs,b+rt)\\
c_2(\Ee(s,t)) &=c+(r-1)(at+bs)+2st{r \choose 2}\\
\chi (\Ee)&=(a+1)(b+1)+r-c-1
\end{align*}
for $(s,t)\in \ZZ^{\oplus 2}$.\\
We use the following notion of stability: a vector bundle $\Ee$ of rank $2$ with $c_1=(0,0)$ is stable if and only if $H^0(\Ee)=0, H^0(\Ee(-1,1))= 0, H^0(\Ee(1,-1))= 0$
(see \cite{LP}).

We denote by $H_*^i(\Ee)$ the graded module $\oplus_{t\in \mathbb Z}H^i(Q, \Ee(t,t))$. We always consider the cohomology of sheaves. $\Ee$ is said to be {\it arithmetically Cohen--Macaulay} ({\it ACM} for short) if $H_*^1(\Ee)=0$.

The Castelnuovo-Mumford criterion states that if $H^i(\Ee(-i,-i))=0$ for all $i>0$, then $\Ee$ is globally generated. We also introduce another criterion for global generatedness which is stronger in some sense.

\begin{theorem}\label{gg}\cite{bm}
For a torsion-free coherent sheaf $\Ee$ on $Q$, the condition
$$H^1(\Ee(-1,0))=H^1(\Ee(0,-1))=H^2(\Ee(-1,-1))=0$$
implies that $\Ee$ is globally generated.
\end{theorem}

\begin{definition}
A globally generated vector bundle $\Ee$ fitting into the sequence
\begin{equation}\label{eqii1}
0\to \Oo_Q(-a,-b)\to \Oo_Q^{\oplus (r+1)} \to \Ee \to 0
\end{equation}
is called to be {\it of maximal type}.
\end{definition}

\begin{remark}\label{ff1}
\begin{enumerate}
\item A vector bundle $\Ee$ of maximal type has $c_1=(a,b)$ and $c_2=c_1^2=2ab$. Thus it achieves the maximal possibility of the second Chern class. Let $\Ff$ be an extension of $\Ee$ by $\Oo_Q$. Since $\Ext^1 (\Oo_Q, \Oo_Q)=0$, so $\Ff$ is of maximal type. In other words, the property of maximal type is preserved under the extension by $\Oo_Q$.
\item Every vector bundle $\Ee$ of maximal type satisfies $H^0(\Ee(-x,-y))=0$ for all pairs of nonnegative integers $(x,y)$ with $x^2+y^2>1$. In fact, in the case of $c_1=(1,0)$ or $(0,1)$, we have $\Ee \cong \Oo_Q^{\oplus (r-1)}\oplus \Oo_Q(c_1)$.
\item  For a rank $r$ globally generated vector bundle $\Ee$ of maximal type with $c_1(\Ee)=(a,b)$, we have $h^0(\Ee)=r+1$ from the sequence (\ref{eqii1}) if $h^1(\Oo_Q(-a,-b))=0$. The condition holds if either $a>0$ and $b>0$ or $(a,b)=(0,1)$ or $(a,b)=(1,0)$.  
\end{enumerate}
\end{remark}

\begin{lemma}\label{mi}
Let us fix a vector bundle $\Ee$ fitting in the exact sequence (\ref{eqii1}) with $h^1(\Oo_Q(-a,-b))=0$.
If $\Ee$ has no trivial factor, then it is indecomposable.
\end{lemma}
\begin{proof}
Let us assume that it is decomposed as $\Ee \cong \Ee_1 \oplus \Ee_2$.  Since
$h^0(\Ee )=\mathrm{rk}(\Ee )+1$ (Remark \ref{ff1}) and $\Ee_1$ and $\Ee_2$ are globally generated, we have $h^0(\Ee _i) = \mathrm{rk}(\Ee_i)$ for some $i$. In particular $\Ee_i$ is trivial and so $\Ee$ has a trivial factor.
\end{proof}

\begin{remark}\label{ff2}
For a vector bundle $\Ee$ fitting in the exact sequence (\ref{eqii1}) with $a=0$ and $r\ge 2$, let us call $\phi : \Oo_Q(0,-b)\to \Oo_Q^{\oplus (r+1)}$ the
map with $\mathrm{coker}(\phi )=\Ee$. Let $\pi : Q\to \PP^1$ denote the projection onto the second factor. Since $a=0$, the map $\phi$ is the pull-back by $\pi$
of a map $v: \Oo _{\PP^1}(-b)\to \Oo _{\PP^1}^{\oplus (r+1)}$ with locally free cokernel. Since $\Ee \cong \pi ^\ast (\mathrm{coker}(v))$ and $r\ge 2$, $\Ee$ is decomposable.
\end{remark}

\begin{lemma}\label{trivial}
Let $\Ff$ be a globally generated sheaf on $Q$. Assume the existence of a non-zero map $f: \Ff \to \Oo _Q$. Then $\Ff$ has $\Oo _Q$ as a factor.
\end{lemma}

\begin{proof}
Since $\Ff$ is globally generated, every non-zero map $\Ff \to \Oo _Q$ is surjective.
Since $\Ff$ is globally generated, there is an integer $n>0$ and a surjection $h: \Oo _Q^{\oplus n} \to \Ff$, say induced by $s_1,\dots ,s_n\in H^0(\mathcal{H}om(\Oo_Q,\Ff))
$. Since $f\circ h$
is surjective, there is $s_i\ne 0$ such that $f\circ s_i\ne 0$. Hence $f\circ s_i: \Oo _Q\to \Oo _Q$ is surjective. Hence $f\circ s_i$ is an isomorphism and so $s_i(\Oo _Q)$ is a direct factor of $\Ff$.
\end{proof}

%%%%%%%%%%%%%%%%%%%%%%%%%%%%%%

\section{Rank two globally generated vector bundle with $c_1< (2,2)$ }
Let us assume that $\Ee$ is a globally generated vector bundle of rank $r$ on $Q$ with the Chern classes $(c_1, c_2)$ and then $\Ee$ fits into the following sequence for general $r-1$ sections of $H^0(\Ee)$ :
\begin{equation}\label{eqa1}
0\to \Oo_Q^{\oplus (r-1)} \to \Ee \to \Ii_Z (c_1) \to 0,
\end{equation}
where $Z$ is a reduced $0$-dimensional subscheme of $Q$ whose length is $c_2$ due to the parts (E) and (F) at page 4 in \cite{am}. Since $\Ii_Z(c_1)$ is globally generated, $Z$ is contained in the complete intersection of two hypersurfaces of bidegree $c_1$. In particular, we have $c_2 \le c_1^2$.

\begin{proposition}\label{prop1}\cite{sierra}
Let $\Ee$ be a globally generated vector bundle of rank $r$ on $Q$ such that $H^0(\Ee(-c_1))\not= 0$, where $c_1$ is the first Chern class of $\Ee$. Then we have
$$\Ee \simeq \Oo_Q^{\oplus (r-1)} \oplus \Oo_Q(c_1).$$
\end{proposition}

As an automatic consequence, $\Ee\simeq  \Oo_Q^{\oplus r}$ is the only globally generated vector bundle of rank $r$ on $Q$ with $c_1=(0,0)$.

\begin{definition}
Let $\ge$ be a partial ordering on $\ZZ^{\oplus 2}$ defined by $(a',b')\ge (a,b)$ if $a' \ge a$ and $b'\ge b$. For a vector bundle $\Ee$ on $Q$, a maximal element with respect to $\ge$ in the following set
$$\{ (a,b)\in \ZZ^{\oplus 2} ~|~ h^0(\Ee (-a,-b))\ne 0 \}$$
is called {\it an index} of $\Ee$.
\end{definition}

\begin{remark}[Due to one of the referees]
In general there may exist more than one indices for a given vector bundle. If $(m,n)$ is an index of a globally generated vector bundle $\Ee$ of rank 2 on $Q$ with $c_1(\Ee)=(a,b)$, then we have $(m,n)\leq (a,b)$. On the other hand, there exists an index of the form $(a,0)$ or $(0,b)$ with $a,b\geq 0$. In general, an arbitrary index $(m,n)$ of $\Ee$ does not have to satisfy $(0,0)\leq (m,n)$. A referee gave the following example. Fixing a generator of $\Lambda^2 V_i$ where $Q\cong \PP V_1 \times \PP V_2$, we get isomorphisms $V_i \cong V_i^\vee$. Let 
$$\tau : \Oo_Q(-1,-1)\to \Oo_Q \otimes (V_1 \otimes V_2)$$
be the canonical injective morphism. Let $\triangle \subset V_1 \otimes V_2$ be a linear subspace of dimension $1$: $\triangle =\CC \delta$. We can also view $\triangle$ as a point in $\PP \mathrm{Hom} (V_1^\vee, V_2)$. Then let 
$$\tau_{\triangle} : \Oo_Q(-1,-1) \to \Oo_Q \otimes ( (V_1 \otimes V_2)/\triangle)=\Oo_Q\otimes \CC^3$$
be the composition of $\tau$ with the quotient $V_1 \otimes V_2 \to (V_1 \otimes V_2)/\triangle$. Then $\tau_{\triangle}$ is injective as a morphism of vector bundles if and only if $\triangle$ does not belong to the quadric of non-bijective morphisms in $\PP \mathrm{Hom}(V_1^\vee, V_2)$, which is the same as the quadric of decomposable elements in $\PP (V_1 \otimes V_2)$. Supposing $\delta$ is bijective, $\Ee_{\triangle}:=\mathrm{coker}(\tau_{\delta})$ is a globally generated vector bundle of rank $2$ with Chern classes $c_1=(1,1)$ and $c_2=2$. From the exact sequence
\begin{equation}
0\to \Oo_Q(-1,-1) \to \Oo_Q\otimes \CC^3 \to \Ee_{\triangle} \to 0
\end{equation}
we have $\mathrm{Hom}(\Oo_Q(-1,1), \Ee_{\triangle}) \cong H^1 (\Oo_Q(0,-2))\cong \CC$. We also deduce immediately that $\mathrm{Hom}(\Oo_Q(0,1), \Ee_{\triangle})=\{0\}$, and an exact sequence
$$\xymatrix{0\ar[r] &\mathrm{Hom}(\Oo_Q(-1,2), \Ee_{\triangle}) \ar[r] & H^1(\Oo_Q(0,-3)) \ar@{=}[d] \ar[r]^-\phi & H^1(\Oo_Q(1,-2))\otimes \CC^3 \ar@{=}[d]\\
&&V_2 & V_1^\vee \otimes (V_1\otimes V_2)/\triangle}$$
The map $\phi : V_2 \to V_1^\vee \otimes (V_1 \otimes V_2)/\triangle$ is the canonical injection and thus we have $\mathrm{Hom}(\Oo_Q(-1,2), \Ee_{\triangle})=\{0\}$. This proves that $(-1,1)$ is an index of $\Ee_{\triangle}$ and so is $(1,-1)$ similarly. Thus we have exact sequences
\begin{align*}
0\to \Oo_Q(-1,1) \to \Ee_{\triangle} \to \Oo_Q(2,0) \to 0,\\
0\to \Oo_Q(1,-1) \to \Ee_{\triangle} \to \Oo_Q(0,2) \to 0,\\
0\to \Oo_Q(-1,1)\oplus \Oo_Q(1,-1) \to \Ee_{\triangle} \to \Oo_C(1,1) \to 0,
\end{align*} 
where $C$ is the curve determined by $\delta$, viewed as an element of $V_1^\vee \otimes V_2^\vee$. Later we show by classification that $\Ee_{\triangle}$ is isomorphic to a pull-back of $T\PP^2(-1)$, the tangent bundle of $\PP^2$ twisted by $-1$, under a linear projection from $Q$ to $\PP^2$. 
\end{remark}

\begin{lemma}\label{prop4}
Let $\Ee$ be a globally generated vector bundle of rank $r$ on $Q$ with $c_1=(a,0)$. If $r>2$, then assume $a\le 2$.Then we have
$$\Ee \cong \bigoplus_{i=1}^r\Oo_Q(a_i,0)$$
where $a_i$'s are nonnegative integers whose sum is $a$.
Similar answer can be given when $c_1=(0,b)$.
\end{lemma}
\begin{proof}
Let us assume that the rank of $\Ee$ is $2$ and $(c,0)$ be an index of $\Ee$. Then $\Ee$ fits into the sequence
$$0\to \Oo_Q(c,0) \to \Ee \to \Ii_Z(a-c,0) \to 0.$$
Since $\Ii_Z(a-c,0)$ is globally generated, so $Z$ must be empty and thus we have $\Ee \cong \Oo_Q(c,0)\oplus \Oo_Q(a-c,0)$. 

Now assume $r>2$ and $a\le 2$. Since $\Ee$ is globally generated, we have an exact sequence
$$0\to \Oo_Q^{\oplus (r-2)} \to \Ee \to \Ff \to 0$$
where $\Ff$ is globally generated and thus we have $\Ff \cong \Oo_Q(b_1,0)\oplus \Oo_Q(b_2,0)$ with $b_1+b_2=a$. If $b_i \le 1$ for
all $i$, then this exact sequence splits. If $b_i\ge 2$ for some $i$, say $b_1\ge 2$, then $b_1=a=2$ and $b_2=0$. Lemma
\ref{trivial} gives $\Ee \cong \Oo _Q\oplus \Gg$ for some $\Gg$. Apply induction on the rank to the bundle $\Gg$.
\end{proof}

Now let us focus on the case of rank $2$. From a general section $s\in H^0(\Ee)$, we have
\begin{equation}\label{eqa3}
0\to \Oo_Q \stackrel{s}{\to} \Ee \to \Ii_Z(c_1) \to 0.
\end{equation}
Since $\Ii_Z(c_1)$ is also globally generated, a general curve of bidegree $c_1$ containing $Z$ is smooth by the Bertini theorem.

%\begin{definition}
%A triple $(a,b,c)\in \ZZ^{\oplus 3}$ is called to be {\it effective} if there exists a globally generated vector bundle of rank $2$ on $Q$ with the Chern classes $c_1=(a,b)$ and $c_2=c$. A non-effective triple is called a {\it gap}.
%\end{definition}

%\begin{remark}
%$ $
%\begin{enumerate}
%\item If $(a,b,c)\in \ZZ^{\oplus 3}$ is effective, then we have $0\le c \le 2ab$.
%\item For $a\ge 0$, a triple $(a,0,c)$ is effective if and only if $c=0$. Similarly $(0,b,c)$ with $b\ge 0$ is effective if and only if $c=0$.
%\end{enumerate}
%\end{remark}

%Note that $\Ee$ is generated by a $4$-dimensional subspace $V\subset H^0(\Ee)$ since we have $\mathrm{rank}(\Ee)+\dim (Q)=4$. So we have an exact sequence
%$$0 \to \Ff^\vee \to V \otimes \Oo_Q \to \Ee \to 0,$$
%where $\Ff$ is a globally generated vector bundle of rank $2$ on $Q$ with the Chern classes $c_1=c_1(\Ee)$ and $c_2=c_1(\Ee)^2-c_2(\Ee)$. $\Ff$ is called the G-dual of $\Ee$. Thus $(a,b,c)$ is effective if and only if $(a,b,2ab-c)$ is effective. Now notice that it is enough to consider the case when $0\le c \le ab$.

\begin{proposition}\label{1.1}
Let $\Ee$ be a globally generated vector bundle of rank $2$ on $Q$ with $c_1=(1,1)$. If $\Ee$ does not split, then $\Ee$ is isomorphic to $\Aa_P:=\phi_P^* (T\PP^2(-1))$, where $\phi _P: Q \to \PP^2$ is the linear projection with the center $P\in \PP^3\setminus Q$.
\end{proposition}

\begin{proof}
By Proposition \ref{prop1}, we can assume that $H^0(\Ee(-1,-1))=0$. If $(1,0)$ is an index of $\Ee$, then we have the sequence
$$0 \to \Oo_Q(1,0) \to \Ee \to \Ii_{Z'} (0,1) \to 0,$$
where $Z'$ is a 0-dimensional subscheme of $Q$ whose length is $c_2-1$. Since $\Ii_{Z'}(0,1)$ is globally generated, $Z'$ is contained in the complete intersection of two hypersurfaces of bidegree $(0,1)$. Thus $Z'$ is an empty set and so $\Ee$ is isomorphic to $\Oo_Q(1,0)\oplus \Oo_Q(0,1)$. We obtain the same answer when $(0,1)$ is an index of $\Ee$.

Now assume that $(0,0)$ is the index of $\Ee$ and then $\Ee$ fits into the sequence (\ref{eqa3}) with $c_1=(1,1)$ and so $\Ii_Z(1,1)$ is globally generated. In particular, $Z$ is contained in the complete intersection of two hyperplane sections of $Q$ and so $\deg (Z)\leq 2$. If $\deg (Z)\leq 1$, then $h^0(\Ee(-1,0))$ or $h^0(\Ee (0,-1))$ would be nonzero, a contradiction. So we have $\deg (Z)=2$ and $Z$ is the complete intersection, not on any ruling of $Q$. From the locally free resolution of $\Ii_Z$
$$0\to \Oo_Q(-2,-2) \to \Oo_Q(-1,-1)^{\oplus 2} \to \Ii_Z \to 0,$$
we obtain the locally free resolution of $\Ee$ :
$$0\to \Oo_Q(-1,-1) \to \Oo_Q^{\oplus 3} \to \Ee \to 0.$$
In particular, $\Ee$ is stable with respect to $\Oo_Q(1,1)$. Conversely, any stable vector bundle $\Ee$ of rank $2$ on $Q$ with the Chern classes $c_1=(1,1)$ and $c_2=2$ is globally generated, since it fits into the sequence
$$0\to \Oo_Q \to \Ee \to \Ii_Z (1,1) \to 0$$
with a $0$-dimensional subscheme $Z$ of length $2$. Note that the moduli space was shown to be isomorphic to $\PP^3 \setminus Q$ in \cite{huh2} using the set of jumping conics. Thus it is enough to prove that $\Aa_P$ is not isomorphic to $\Aa_O$ if $P\ne O$. It is true since the set of jumping conics of $\Aa_P$ is the hyperplane in $(\PP^3)^\vee$ corresponding to the point $P\in \PP^3$.
\end{proof}

\begin{remark}
Without confusion, we simply denote $\Aa_P$ by $\Aa$. The bundle $\Aa$ is of maximal type and so its extension by $\Oo_Q^{\oplus (r-2)}$ is also of maximal type. Since $h^1(\Aa^\vee)=1$, so the only possible extension of $\Aa$ by $\Oo_Q^{\oplus (r-2)}$ without any trivial factor happens when $r=3$. Later we will see that such bundle is always isomorphic to $T\PP^3(-1)|_Q$ for any $P\in \PP^3\setminus Q$.
\end{remark}

\begin{proposition}\label{prop1.2}
Let $\Ee$ be a globally generated vector bundle of rank $2$ on $Q$ with $c_1=(1,2)$. If $\Ee$ does not split, then it is one of the following:
\begin{enumerate}
\item $(c_2=2) :  0\to \Oo_Q \to \Oo_Q(1,0)\oplus \Oo_Q(0,1)^{\oplus 2} \to \Ee \to 0$
\item $(c_2=3) : 0\to \Oo_Q(-1,-1) \to \Oo_Q(0,1)\oplus \Oo_Q^{\oplus 2} \to \Ee \to 0$
\item $(c_2=4) : 0\to \Oo_Q(-1,-2) \to \Oo_Q^{\oplus 3} \to \Ee \to 0$
\end{enumerate}
Similar answer can be given when $c_1=(2,1)$.
\end{proposition}
\begin{proof}
$\Ee$ fits into the sequence (\ref{eqa3}) with $c_1=(1,2)$. Since $\Ii_Z(1,2)$ is globally generated, $Z$ is contained in a complete intersection of two hypersurfaces of bidegree $(1,2)$ in $Q$. In particular, $\deg (Z)$ is at most 4. If $\deg (Z)=4$, then $Z$ is the complete intersection and so $\Ii_Z$ admits a locally free resolution:
$$0\to \Oo_Q(-2,-4) \to \Oo_Q(-1,-2)^{\oplus 2} \to \Ii_Z \to 0.$$
It induces a locally free resolution (3) in the list.

Assume that $\deg (Z)\leq 3$ and so $h^0(\Ii_Z(1,1))\not= 0$. In particular, we have $h^0(\Ee(0,-1))\not= 0$ and so an index of $\Ee$ is $(a,b)$ with $b\ge 1$ and $\Ee$ fits into the sequence
$$0\to \Oo_Q(a,b) \to \Ee \to \Ii_{Z'}(1-a,2-b) \to 0.$$
If either $a=1$ or $b=2$, then $Z'$ is empty and so we have $\Ee \simeq \Oo_Q(a,b)\oplus \Oo_Q(1-a,2-b)$. Now let us assume that $(a,b)=(0,1)$. Since $\Ii_{Z'}(1,1)$ is globally generated, $Z'$ is contained in a complete intersection of two hypersurfaces of bidegree $(1,1)$ in $Q$ and so $\deg (Z')\leq 2$. If $Z'$ is an empty set, then $\Ee$ is isomorphic to $\Oo_Q(0,1)\oplus \Oo_Q(1,1)$. If $\deg (Z')=2$, then $Z'$ is the complete intersection and so $\Ii_{Z'}$ admits a locally free resolution:
$$0\to \Oo_Q(-2,-2) \to \Oo_Q(-1,-1)^{\oplus 2} \to \Ii_{Z'} \to 0.$$
It induces a locally free resolution (2) in the list. If $\deg (Z')=1$, then $Z'$ is a complete intersection of two hypersurfaces of bidegree $(1,0)$ and $(0,1)$ and so $\Ii_{Z'}$ admits a locally free resolution:
$$0\to \Oo_Q(-1,-1)\to \Oo_Q(-1,0)\oplus \Oo_Q(0,-1) \to \Ii_{Z'} \to 0.$$
It induces the resolution (1) in the list.
\end{proof}

\begin{remark}
When we combine the exact sequence
$$0\to \Oo_Q(0,-1) \to \Oo_Q^{\oplus 2} \to \Oo_Q(0,1) \to 0$$
to the resolution (2) in the list, we have
\begin{equation}
0\to \Oo_Q(-1,-1)\oplus \Oo_Q(0,-1) \to \Oo_Q^{\oplus 4} \to \Ee \to 0.
\end{equation}
Similarly, the resolution (1) becomes the sequence,
\begin{equation}
0\to \Oo_Q(0,-1)^{\oplus 2} \oplus \Oo_Q(-1,0)\to \Oo_Q^{\oplus 5} \to \Ee \to 0.
\end{equation}
\end{remark}

\section{Case $c_1=(2,2)$}

Let $\mathfrak{M}(k)$ be the moduli space of stable vector bundles of rank $2$ on $Q$ with the Chern classes $c_1=(2,2)$ and $c_2=k$ with respect to the ample line bundle $\Oo_Q(1,1)$. By the Bogomolov inequality, the space is empty if $k<2$. In \cite{Soberon}, it was proven to be a smooth, rational and irreducible variety of dimension $4k-11$.

In this section, we assume that $\Ee$ is a globally generated vector bundle of rank $2$ on $Q$ with $c_1(\Ee)=(2,2)$. If $(2,2)$ is the index of $\Ee$, then we have $\Ee \cong \Oo_Q\oplus \Oo_Q(2,2)$ by Proposition \ref{prop1}. If $(2,1)$ is an index of $\Ee$, then we have
$$0\to \Oo_Q(2,1) \to \Ee \to \Ii_Z(0,1) \to 0.$$
Since $\Ii_Z(0,1)$ is globally generated, so $Z$ is an empty set and we have $\Ee \cong \Oo_Q(2,1)\oplus \Oo_Q(0,1)$. In general, if $(a,b)$ is an index of $\Ee$ with $a+b\ge 3$, then we have $\Ee \cong \Oo_Q(a,b)\oplus \Oo_Q(2-a,2-b)$.

\begin{lemma}\label{nost}
If $(0,2)$ is an index of a non-splitting bundle $\Ee$, then $\Ee$  arises in the following  extension:
 $$0\to \Oo_Q(0,2) \to \Ee \to \Oo_Q(2,0) \to  0.$$
A similar result holds for the case of index $(2,0)$.
\end{lemma}
\begin{proof}
We have the following exact sequence
$$0\to \Oo_Q(0,2) \to \Ee \to \Ii_Z(2,0) \to 0$$
and since $\Ii_Z(2,0)$ is globally generated, we have $\deg (Z)=0$. So $\Ii_Z$ is a line bundle and we get the claimed extension. Since $\dim(\Ext^1( \Oo_Q(2,0), \Oo_Q(0,2)))=3$ the bundle $\Ee$ may be indecomposable.
\end{proof}

\begin{lemma}\label{le1}
If $(1,1)$ is an index of non-splitting $\Ee$, then $\Ee$ satisfies one of the following:
\begin{enumerate}
\item $(c_2=3) : 0\to \Oo_Q \to \Oo_Q(1,1)\oplus \Oo_Q(1,0)\oplus \Oo_Q(0,1) \to \Ee \to 0$\item $(c_2=4) : 0\to \Oo_Q(-1,-1) \to \Oo_Q^{\oplus 2} \oplus \Oo_Q(1,1) \to \Ee \to 0$.
\end{enumerate}
\end{lemma}
\begin{proof}
We have the following exact sequence
$$0\to \Oo_Q(1,1) \to \Ee \to \Ii_Z(1,1) \to 0$$
and since $\Ii_Z(1,1)$ is globally generated, we have $\deg (Z)\le 2$. If $\deg (Z)=0$, then we have $\Ee \cong \Oo_Q(1,1)^{\oplus 2}$. If $\deg (Z)=1$, say $Z=\{P\}$, then we have
$$h^1(\Ii_P)=h^2(\Ii_P(-1,-1))=0$$
and so $\Ii_P(1,1)$ is globally generated due to Castelnuovo-Mumford criterion. Since $h^1(\Oo_Q(1,1))=0$, so $\Ee$ is globally generated. Note that $\Ii_P$ admits the following resolution
$$0\to \Oo_Q(-1,-1) \to \Oo_Q(-1,0)\oplus \Oo_Q(0,-1) \to \Ii_P\to 0$$
and so $\Ee$ admits the following resolution (1). Similarly when $\deg (Z)=2$, we obtain the resolution (2).
\end{proof}

If $\Ee$ has an index $(a,b)$ with $a+b\le 1$, then $\Ee$ fits into the sequence
$$0\to \Oo_Q(a,b) \to \Ee \to \Ii_{Z}(2-a,2-b) \to 0.$$
Since $\Ii_Z(2-a,2-b)$ is globally generated, so we have $\deg (Z)\le 2(2-a)(2-b)$. Thus we have $c_2=c_2(\Ee)=\deg (Z)+a(2-b)+b(2-a)\le 8-2(a+b)$. Since $\Ee$ is stable, we have $\Ee \in \mathfrak{M}(c_2)$.

\begin{lemma}\label{le2}
Let $\Ee$ be a globally generated vector bundle of rank $2$ on $Q$ with $c_1=(2,2)$ and index $(1,0)$. Then $\Ee$ is in $\mathfrak{M}(c_2)$ with $4\le c_2 \le 6$. In fact, we have the following:
\begin{enumerate}
\item Every bundle $\Ee \in \mathfrak{M}(4)$ is globally generated.
\item A general bundle in $\mathfrak{M}(5)$ is globally generated.
\item If $c_2(\Ee)=6$, then it fits into the sequence
$$0\to \Oo_Q(-1,-2) \to \Oo_Q^{\oplus 2} \oplus \Oo_Q(1,0) \to \Ee \to 0.$$
\end{enumerate}
Similar answer can be given when $(0,1)$ is an index of $\Ee$.
\end{lemma}
\begin{proof}
Let $\Ee$ be an arbitrary bundle with index $(1,0)$ and so fitting into the sequence
\begin{equation}\label{11}
0\to \Oo_Q(1,0) \to \Ee \to \Ii_Z(1,2) \to 0
\end{equation}
with $\deg (Z)\le 4$. If $\deg (Z)\le 1$, then we obtain $h^0(\Ii_Z(0,1))\ne 0$ and so $h^0(\Ee(-1,-1))\ne 0$. It is a contradiction to the assumption that $(1,0)$ is an index. Thus we have $\deg (Z)\ge 2$. Assume that $\deg (Z)=2$. If $H^0(\Ii_Z(0,1)) \ne 0$, then $h^0(\Ee(-1,-1))\ne 0$ and it is contradicting to the assumption that $(1,0)$ is an index. Thus $h^0(\Ii_Z(0,1))=0$ and so $h^0(\Ii_Z(0,2))=1$. It implies that $h^1(\Ii_Z(0,2))=0$. We can also prove that $h^1(\Ii_Z(1,1))=h^2(\Ii_Z(0,1))=0$. Thus $\Ii_Z(1,2)$ is globally generated by Theorem \ref{gg}. From the exact sequence, we also obtain the stability of $\Ee$. Now let us pick a bundle $\Ee$ from $\mathfrak{M}(4)$. From the stability condition, we have
$$h^0(\Ee (-1,-1))=h^0(\Ee(0,-2))=h^0(\Ee(-2,0))=0.$$
Since $\chi (\Ee(-1,0))=2$ and $h^2(\Ee(-1,0))=0$, we have the sequence (\ref{11}) with $Z$ of length $2$ such that $h^0(\Ii_Z(0,1))=0$. As before, it implies the globally generatedness of $\Ee$.

Let us assume that $\deg (Z)=3$ and in particular we have $c_2(\Ee)=5$. Let $C\subset Q$ be a smooth curve of type $(1,2)$ and then it is a rational normal curve of $\PP ^3$. Let $Z'\subset C$ be any zero-dimensional subscheme of degree $5$. Since $C\cong \PP^1$ and $\deg (\Ii_{Z',C}(2))=1$, $\Ii_{Z',C}(2)$ is globally generated. It implies that the sheaf $\Ii_{Z'}(2,2)$ is globally generated since the restriction map $H^0(\Oo_Q(2,2)) \to H^0(C, \Oo_C(2))$ is surjective. Thus the bundles fitting into the sequence
$$0\to \Oo_Q \to \Ee \to \Ii_{Z'}(2,2) \to 0$$
are globally generated.

 If $\deg (Z)=4$, then $Z$ is a complete intersection of two hypersurfaces of bidegree $(1,2)$ and thus we obtain the resolution (3) by a routine way.
\end{proof}

\begin{remark}\label{re44}$ $
\begin{enumerate}
\item For $\Ee \in \mathfrak{M}(k)$ with $k\in \{4,5\}$, we can observe that $(1,0)$ and $(0,1)$ are the indices of $\Ee$ simultaneously, but not $(1,1)$.
\item By twisting with $\Oo_Q(-1,-1)$, we can identify $\mathfrak{M}(4)$ with $\mathfrak{M} ((0,0),2)$, the moduli space of stable vector bundles of rank $2$ on $Q$ with $c_1=(0,0)$ and $c_2=2$. Le Potier \cite{LP} analyzes the restriction to the quadric $Q$ of null correlation bundles $\Nn$. Let $\mathfrak{M}_{\mathbb
P^3}^0(0,1)$ be the open subset of $\mathfrak{M}_{\mathbb P^3}(0,1)$ consisting of all bundles $\Nn$ such that $\Nn |_{\cal Q}$ is
stable on $Q$. He shows that the restriction gives an \`etale quasi-finite morphism from $\mathfrak{M}_{\mathbb P^3}^0(0,1)$ onto an
open proper subset $U \subset \mathfrak{M}((0,0),2)$. The generic bundle $\Ee$ of $U$ has a twin pair (a Tjurin pair) of null correlation bundles restricting to it, while there are bundles $\Ee$ in $U$ with a unique null correlation bundle
restricting to it. In \cite{Soberon} Soberon-Chavez compactifies $\mathfrak{M}((0,0),2)$ using only the non-split and non-stable bundles described in Lemma \ref{nost}. In \cite{mr} is given an example of a bundle in  $\mathfrak{M}((0,0),2)$ but not in $U$, i.e. a stable bundle which is not the restriction of a null correlation bundle.
\item $\mathfrak{M} (5)$ is invariant by the involution $\sigma : Q\to Q$
which exchanges the two rulings of $Q$, i.e. $\sigma ^\ast (\Ee )\in \mathfrak{M} (5)$
for each $\Ee \in \mathfrak{M} (5)$. For a fixed $\Ee \in \mathfrak{M} (5)$, we obviously have $c_2(\sigma ^\ast (\Ee ))=5$.
Since $\sigma ^\ast (\Oo _Q(2,2)) \cong \Oo _Q(2,2)$, we have $c_1(\sigma ^\ast (\Ee )) = (2,2)$.
Assume that $\sigma ^\ast (\Ee )$ is not stable, i.e. assume the existence of
a line bundle $\Ll = \Oo _Q(a,b) \subset \sigma ^\ast (\Ee )$ such that $a+b \ge 4$. We have
$\Oo _Q(b,a) \cong \sigma ^\ast (\Ll) \subset \sigma ^\ast (\sigma ^\ast (\Ee )) \cong \Ee$.
Hence $\Ee$ is not stable, a contradiction.
\end{enumerate}
\end{remark}

When the index is $(0,0)$ we need the following definition:
\begin{definition}
Let $Z\subset Q$ be a locally complete intersection (l.c.i.) $0$-dimensional subscheme. For $(a,b)\in \ZZ^{\oplus 2}$, $Z$ is said to satisfy {\it Cayley-Bacharach} for curves of bidegree $(a,b)$, simply $\mathrm{CB}(a,b)$, if any curve of bidegree $(a,b)$ containing a subscheme of $Z$ with colength 1 contains $Z$.
\end{definition}
\begin{proposition}\cite{gh}
For a l.c.i. $0$-dimensional subscheme $Z\subset Q$, there exists an exact sequence
$$0\to \Oo_Q \to \Ee \to \Ii_Z(a,b) \to 0$$
where $\Ee$ is a vector bundle of rank $2$ if and only if $Z$ satisfies $\mathrm{CB}(a-2,b-2)$.
\end{proposition}

\begin{lemma}\label{le3}
If $(0,0)$ is the index of $\Ee$, then we have one and only one of the following cases:
\begin{enumerate}
\item $\Ee\in \mathfrak{M}(6)$ with index $(0,0)$.
\item $\Ee$ is of maximal type.
\end{enumerate}
\end{lemma}
\begin{proof}
$\Ee$ fits into the sequence
$$0\to \Oo_Q \to \Ee \to \Ii_Z(2,2) \to 0$$
with $\deg (Z)\le 8$. If $\deg (Z)\le 5$, then we have $h^0(\Ii_Z(1,2))\ne 0$ and so $(0,0)$ cannot be the index of $\Ee$. Thus we have $\deg (Z)\ge 6$. If $\deg (Z)=8$, then $Z$ is a complete intersection of two hypersurfaces of bidegree $(2,2)$. Since $\Ii_Z$ admits the resolution $$0\to \Oo_Q(-4,-4) \to \Oo_Q(-2,-2)^{\oplus 2} \to \Ii_Z \to 0,$$
so we have the resolution
$$0\to \Oo_Q(-2, -2) \to \Oo_Q^{\oplus 3} \to \Ee \to 0$$
and in particular, $\Ee$ is of maximal type. The case of $\deg (Z)=7$ is not possible since $\Ii_Z(2,2)$ is not globally generated (see Proposition \ref{length7}).

Now let us assume that $\deg (Z)=6$. Let $\mathrm{Hilb}^6(Q)$ be the Hilbert scheme of all $0$-dimensional subschemes of $Q$ with degree $6$. Since $H^0(\Ii_{Z'}(0,0))=0$ for all $Z'\subset Z$ with $\deg (Z')>0$, the {\it Cayley-Bacharach} condition is satisfied and so we can associate a vector bundle of rank $2$ with $c_1=(2,2)$ and $c_2=6$ to each $Z\in \mathrm{Hilb}^6(Q)$. Let us define the subset $\Uu$ to be the set of all $Z$ such that $\Ii_Z(2,2)$ is globally generated and so is its associated vector bundle of rank $2$. Since $\mathrm{Hilb}^6(Q)$ is smooth and irreducible with dimension $12$, so $\Uu$ is non-empty, irreducible and of dimension $12$.

Let us also define a subset $\Vv\subset \mathrm{Hilb}^6(Q)$:
$$\Vv =\{ Z \in \mathrm{Hilb}^6(Q) ~|~ H^0(\Ii_Z(1,2))=H^0(\Ii_Z(2,1))=0\}.$$
For $Z\in \Vv$, we have $H^1(\Ii_Z(1,2))=H^1(\Ii_Z(2,1))=0$ and $H^2(\Ii_Z(1,1))=0$. Thus $\Ii_Z(2,2)$ is globally generated by Theorem \ref{gg} and so we have $\Vv \subset \Uu$. Note that the union of $6$ general points is contained in $\Vv$. From its definition, a vector bundle $\Ee$ is associated to $Z\in \Vv$ if and only if the index of $\Ee$ is $(0,0)$.

For each $Z\in \Vv$, the dimension of $\Ext^1 (\Ii_Z(2,2), \Oo_Q)$ is $5$ and so the set of all globally generated vector bundles of rank $2$ with $c_1=(2,2)$ and $c_2=6$ is parametrized by an irreducible variety. Each such a bundle is obviously stable. For each $Z\in \Vv$, we have $h^1(\Ii_Z(1,2))=0$ and so $h^1(\Ii_Z(2,2))=0$. Thus we have $h^0(\Ee)=4$ and $h^1(\Ee)=0$.
\end{proof}

\begin{proposition}\label{length7}
For $0$-dimensional subschemes $Z$ of $Q$ with length $7$, $\Ii_Z(2,2)$ is not globally generated.
\end{proposition}

Before proving the proposition, let us collect some materials on the residual scheme.

\begin{definition}
Let $A\subset Q$ be a $0$-dimensional subscheme and $D\subset Q$ be an effective divisor of type $(u,v)\in \ZZ^{\oplus 2}$. {\it The residual scheme of $A$ with respect to $D$} is defined as the closed subscheme of $Q$ with $(\Ii_A : \Ii_D)$ as its ideal sheaf. It is denoted by $\mbox{Res}_D(A)$.
\end{definition}

\begin{remark}
From the definition, we have
\begin{itemize}
\item $\mbox{Res}_D(A) \subseteq A$ and
\item $\deg (A) = \deg (A\cap D) +\deg (\mbox{Res}_D(A))$.
\end{itemize}
For all $m, n\in \ZZ$ we also have an exact sequence
\begin{equation}\label{eqa10}
0 \to \Ii _{\mbox{Res}_D(A)}(m-u,n-v) \to \Ii _A(m,n)\to  \Ii _{D\cap A,D}(m,n) \to 0
\end{equation}
\end{remark}

\begin{lemma}\label{a1}
For $L_1\in \vert \Oo _Q(1,0)\vert$ and $L_2\in \vert \Oo _Q(0,1)\vert$, let us take a zero-dimensional subscheme $A\subset L_1\cup L_2$ such that $\deg (A) \le 3$ and $\deg (A\cap L_i) \le 2$ for
all $i$. Then we have $h^1(\Ii _A(1,1)) =0$.
\end{lemma}
\begin{proof}
Since $\deg (A) \le \deg (A\cap L_1)+\deg (A\cap L_2)$, there is $i\in \{1,2\}$ such that $\deg (L_i\cap A) =2$. Without loss of generality, let us assume that $i=1$.
Since $\deg (L_1\cap A) =2$, we have $h^1(\Ii _{A\cap L_1}(1,1)) =0$. Since $\deg (\mbox{Res}_{L_1}(A)) \le 1$, we have $h^1(\Ii _{\mbox{Res}_{L_1}(A)}(0,1)) =0$ and then we can apply the sequence (\ref{eqa10}).
\end{proof}

\begin{lemma}\label{a2}
For $L_1\in \vert \Oo _Q(1,0)\vert$ and $L_2\in \vert \Oo _Q(0,1)\vert$, let us take a zero-dimensional subscheme $A\subset L_1\cup L_2$ such that $\deg (A\cap L_i) \le 2$ for
all $i$. Then we have $h^1(\Ii _A(2,2)) =0$.
\end{lemma}

\begin{proof}
Since $\deg (A) \le \deg (A\cap L_1)+\deg (A\cap L_2)$, we have $\deg (A)\le 4$. Lemma \ref{a1} gives the
case of $\deg (A)\le 3$. Assume $\deg (A)=4$. We have $h^1(\Ii _{\mbox{Res}_{L_1}}(1,2)) =0$ because $\mbox{Res}_{L_1}(A)\subset L_2$, $L_2\cong \mathbb {P}^1$
and $\deg (\mbox{Res}_{L_1}(A) \cap L_2) \le \deg (A\cap L_2) \le 2$. We then apply the sequence (\ref{eqa10}).
\end{proof}

\begin{proof}[Proof of Proposition \ref{length7}]
Let us take $Z\subset Q$ with $\deg (Z)=7$ and assume that $\Ii _Z(2,2)$ is globally generated. In particular, we have $\deg (Z\cap L) \le 2$ for all lines $L\subset Q$.
Since $\deg (Z) < 8$, $Z$ is not a complete intersection and so $h^0(\Ii _Z(2,2)) \ge 3$. It implies that $h^1(\Ii _Z(2,2)) >0$.

Let us take a curve $E\in \vert \Oo _Q(1,1)\vert$ such that $w:= \deg (A\cap E) $ is maximal. Assume for the moment that $E$ is irreducible. Since $\Ii _Z(2,2)$ is globally generated, we have $w \le 4$. If $E$ is reducible, say $E =L_1\cup L_2$, then the same is
true, because $\deg (A\cap (L_1\cup L_2)) \le \deg (A\cap L_1)+\deg (A\cap L_2)$. In both cases we have $h^1(\mathcal {I}_{E\cap Z}(2,2)) =0$; if $E$ is reducible, this is Lemma \ref{a2}. Hence the sequence (\ref{eqa10}) gives $h^1(\Ii _{\mbox{Res}_E(Z)}(1,1)) > 0$.

\quad (a) Let us assume first that $w =4$ and so $\deg (\mbox{Res}_E(Z)) =3$. We have $\deg (\mbox{Res}_E(Z)\cap L) \le \deg (Z\cap L) \le 2$ for each line $L$. First assume the existence of a line $R_1$, say of type $(1,0)$, such that $\deg (R_1\cap  \mbox{Res}_E(Z)) =2$. We have $h^1(R_1,\Ii _{R_1\cap  \mbox{Res}_E(Z),R_1}(1,1)) =0$ and $\deg (\mbox{Res}_{R_1}(\mbox{Res}_E(Z)))=1$. Hence $h^1(\Ii _{\mbox{Res}_{R_1}(\mbox{Res}_E(Z))}(0,1))=0$. We apply the sequence (\ref{eqa10}) to get a contradiction. Now assume $\deg (\mbox{Res}_E(Z)\cap L)
\le 1$ for each line. Since $h^0(\Oo _Q(1,1)) =4$, there is a smooth conic $D$ such that $\deg (D\cap \mbox{Res}_E(Z)) \ge 2$. Since $\deg (\mbox{Res}_E(Z)) =3$, we have $\deg (\mbox{Res}_D(\mbox{Res}_E(Z))) \le 1$ and so $h^1(\Ii _{\mbox{Res}_D(\mbox{Res}_E(Z))})=0$. Since $D$ is a smooth conic and $\deg (D\cap \mbox{Res}_E(Z)) \le 3$, we have $h^1(D,\Ii _{\mbox{Res}_E(Z),D}(1,1)) =0$, contradicting (\ref{eqa10}).

\quad (b) Now assume that $w=3$ and so $\deg (\mbox{Res}_E(Z)) =4$. Let us take $F\in \vert \Oo _Q(1,1)\vert$ such that $z:= \deg (F\cap \mbox{Res}_E(Z))$
is maximal. Since $3 \le z \le w =3$, we have $z=3$ and so $h^1(F,\Ii _{F\cap \mbox{Res}_E(Z),F}(1,1)) =0$, using Lemma \ref{a1} if $F$ is reducible.
From (\ref{eqa10}), we get $h^1(\Ii _{\mbox{Res}_F(\mbox{Res}_E(Z)})) >0$. But we have $\deg (\mbox{Res}_F(\mbox{Res}_E(Z)))=1$ and this is absurd.
\end{proof}

%%%%%%%%%%%%%%%%%%%%%%%%%%%%%%

\section{Higher Rank Case}

Let $\Ee$ be a globally generated vector bundle of rank $r \ge3$ on $Q$ and it fits into the following sequence
\begin{equation}\label{heqa}
0\to \Oo_Q^{\oplus (r-2)} \to \Ee \to \Ff \to 0
\end{equation}
where $\Ff$ is a globally generated vector bundle of rank $2$ on $Q$ with the Chern classes $c_i(\Ff)=c_i(\Ee)$ for $i=1,2$. Conversely, if $\Ff$ is a globally generated vector bundle of rank $2$, then any coherent sheaf $\Ee$ fitting into the sequence (\ref{heqa}) is globally generated since $h^1(\Oo_Q)=0$. Note that $h^0(\Ee)=h^0(\Ff)+r-2$.

\begin{lemma}
If $\Ee$ has no trivial summand, then we have
$$\mathrm{rank} (\Ee) \le h^1(\Ff^\vee)+2.$$
\end{lemma}
\begin{proof}
If $r:=\mathrm{rank}(\Ee)> h^1(\Ff^\vee)+2$, then $\Ee$ is given by extension classes $e_1, \cdots, e_{r-2} \in H^1(\Ff^\vee)$ and they are linearly dependent. By changing a basis of the trivial bundle $\Oo_Q^{\oplus (r-2)}$, we reduce to the case of $e_{r-2}=0$ and so we have $\Ee \cong \Gg \oplus \Oo_Q$ for some vector bundle $\Gg$ of rank $r-1$. In particular, $\Ee$ has a trivial factor.
\end{proof}

\begin{corollary}
There exists an indecomposable and globally generated vector bundle on $Q$ with $c_1=(2,2)$ only if the rank is at most $8$.
\end{corollary}
\begin{proof}
From the classification of such vector bundles of rank $2$, we have $h^1(\Ff^\vee) \le 6$ and thus the rank of $\Ee$ is at most $8$.
\end{proof}

The goal of this section is to describe the possible rank and Chern classes with which indecomposable vector bundles on $Q$ exist. The main result is the following:

\begin{theorem}\label{hr}
There exists a globally generated and indecomposable vector bundle of rank $r\ge 3$ with the Chern classes $(c_1, c_2)\in \ZZ^{\oplus 3}$ with $c_1=(a,b)\le (2,2)$ and $a\le b$ if and only if
\begin{align*}
(c_1, c_2;r)\in \{  &(1,1,2;3), (1,2,3;3),(1,2,4;r=3,4,5),(2,2,4;3)\\
&(2,2,5;3),(2,2,6;r=3,4,5), (2,2,8;r=3,4,5,6,7,8) \}.
\end{align*}
\end{theorem}

\begin{proposition}
Let $\Ee$ be a globally generated and indecomposable vector bundle of rank $r\ge 3$ with $c_1=(1,1)$. Then we have
$$\Ee \cong T\PP^3(-1)_{|_Q}.$$
\end{proposition}
\begin{proof}
Note that $h^1(\Oo_Q(-a,-b))=0$ for $a,b\in \{0,1\}$ and so the only possibility for $\Ff$ is $\Aa_P$ from Proposition \ref{1.1}. Since $h^1(\Aa_P^\vee)=h^1(\phi_P^*(\Omega_{\PP^2}^1(1)))=1$, so there exists a non-trivial extension
$$0\to \Oo_Q \to \Bb_P \to \Aa_P \to 0.$$
For any two points $P,O\in \PP^3\setminus Q$, there exists an automorphism $g\in \mathrm{Aut}(Q)$ such that $g^*(\Aa_P)\cong \Aa_O$ because $\mathrm{Aut}(Q)$ acts transitively on the set of all points of $\PP^3\setminus Q$ and any pull-back of twisted tangent bundle is uniquely determined by its center of projection. Thus $g^*(\Bb_P)\cong \Bb_O$. Conversely, since $T\PP^3(-1)|_Q$ is globally generated, so it fits into an exact sequence
$$0\to \Oo_Q \to T\PP^3(-1)|_Q \to \Ff \to 0$$
for some globally generated vector bundle $\Ff$ of rank 2 with $c_1(\Ff)=(1,1)$. From the classification of such bundles, we have $\Ff \cong \Aa_P$ for some $P$. Thus we have $T\PP^3(-1)|_Q \cong \Bb_P$ for each $P$.
\end{proof}

\begin{proposition}
Let $\Ee$ be a globally generated and indecomposable vector bundle of rank $r\ge 3$ with $c_1=(1,2)$. Then $\Ee$ is of maximal type or has $c_2=3$ with $\mathrm{rank}(\Ee)=3$.
\end{proposition}
\begin{proof}
It is obvious that the splitness of $\Ff$ in the sequence (\ref{heqa}) implies the splitness of $\Ee$ in the case of $c_1=(1,2)$. If $\Ff$ admits the sequence (1) in Proposition \ref{prop1.2}, then we have $h^1(\Ff^\vee)=1$, i.e. there exists a unique non-trivial extension of $\Ff$ by $\Oo_Q$. It is clearly $\Oo_Q(1,0)\oplus \Oo_Q(0,1)^{\oplus 2}$ and, in particular, $\Ee$ splits.

If $\Ff$ is of maximal type, then we have $h^1(\Ff^\vee)=3$ and thus there exist indecomposable vector bundles of rank up to $5$ with $c_2=4$.

Let us assume that $\Ff$ admits the sequence (2) in Proposition \ref{prop1.2}, i.e. $c_2(\Ff)=3$. Since $h^1(\Ff^\vee)=2$, so the rank of $\Ee$ must be at most $4$. Let us assume that the rank of $\Ee$ is $4$, i.e. we have
$$0\to \Oo_Q^{\oplus 2} \to \Ee \to \Ff \to 0.$$
From the sequence (2), we have $h^0(\Ff(-1))=0$ and $h^1(\Ff(-2))=2$. Thus we have $h^0(\Ee(-1))=h^1(\Ee(-2))=0$. By Theorem 6.7 in \cite{AO} with $S''=\Oo_Q(0,-1)$ and $t=0$, we have $h^1(\Ee(-2,-1))=h^1(\Ff(-2,-1))=1$ and so $\Ee$ has $\Oo_Q(0,1)$ as its direct summand. Thus we have $\Ee \cong \Oo_Q(0,1)\oplus \Gg$ where $\Gg$ is a globally generated vector bundle of rank $3$ with $c_1=(1,1)$ and $c_2=2$. In particular, it is decomposable, a contradiction.

Let us assume $\mathrm{rank}(\Ee)=3$ and so it is an extension of $\Ff$ by $\Oo_Q$. In particular, we have $h^1(\Ee^\vee)=1$ and so there exists a non-trivial extension $\Gg$ of $\Ee$ by $\Oo_Q$. Since $\Gg$ is a globally generated vector bundle of rank $4$ with $c_1=(1,2)$ and $c_2=3$, so we have $\Gg \cong \Oo_Q(0,1)\oplus \Uu$ with either $\Uu \cong T\PP^3(-1)|_Q$ or $\Uu$ has a trivial factor; the latter case cannot occur because $\Gg$ has no trivial factor. Note that $h^0(\Ee(0,-1))=h^0(\Ff(0,-1))=1$. Let us assume that $\Ee$ is decomposable and then $\Oo_Q(0,1)$ must be a direct factor, i.e. we have $\Ee \cong \Oo_Q(0,1)\oplus \Gg$ where $\Gg$ is globally generated with $c_1=(1,1)$ and $c_2=2$. From the classification, we have $\Gg \cong \Aa_P$ for some $P\in \PP^3\setminus Q$.

Let us take a general section of $\Oo _Q(0,1)\oplus T\PP^3(-1)|_Q$. Hence
we get an exact sequence
\begin{equation}\label{eqa1}
0 \to \Oo _Q\stackrel{u}{\to} \Oo _Q(0,1)\oplus T\PP^3(-1)|_Q \stackrel{\psi}{\to} \Ee \to 0
\end{equation}

Assume that $\Ee$ is decomposable, i.e. $\Ee \cong \Oo _Q(0,1) \oplus \Aa _P$ for some $P$. Hence $\psi$ is given by a $2\times 2$
matrix of maps. Since $h^0(\Aa _P(0,-1)) = h^0({\Omega _{\PP^3}}|_Q(1,0)) = 0$
, using the exact sequence
$$0 \to {\Omega _{\PP^3}}|_Q (1,0) \to \Oo _Q(0,-1)^{\oplus 4} \to \Oo _Q(1,0)\to 0,$$
$\psi$ is the diagonal matrix associated to $\psi _1: \Oo _Q(0,1)\to \Oo _Q(0,1)$
and $\psi _2: T\PP^3(-1)|_Q \to \Aa _P$. $\psi _1$ must be the multiplication by a non-zero-constant,
while $\psi _2$ must be surjective and have $\Oo _Q$ as its quotient. Hence $\mbox{ker}(\psi)$ is contained
in the factor $\{0\}\oplus T\PP^3(-1)|_Q$, not just isomorphic to a subsheaf of that factor and so $\mbox{Im}(u)$ is contained
in the factor $\{0\}\oplus T\PP^3(-1)|_Q$ , again not just isomorphic to a subsheaf of that factor. This is not true for a general $u$ and thus there exists an indecomposable and globally generated vector bundle of rank $3$ with $c_1=(1,2)$ and $c_2=3$.
\end{proof}

Let us assume that $\Ee$ has the first Chern class $c_1=(2,2)$. If the associated vector bundle $\Ff$ of rank $2$ splits, then we have $h^1(\Ff^\vee)=0$ except when $\Ff \cong \Oo_Q(2,0)\oplus \Oo_Q(0,2)$.

\begin{lemma}
There exists an indecomposable extension $\Ee$ of $\Oo_Q(2,0)\oplus \Oo_Q(0,2)$ by trivial factors if and only if $\mathrm{rank}(\Ee)=3$.
\end{lemma}

\begin{proof}
Let $\Ff$ be $\Oo_Q(2,0)\oplus \Oo_Q(0,2)$. Since we have $h^1(\Ff^\vee)=2$ for a bundle $\Ff$ in the sequence (\ref{heqa}), so $\Ee$ has a trivial factor if $\mathrm{rank}(\Ee)\ge 5$.  If $\Ee$ is a vector bundle arising in a non-trivial extension $$0\to \Oo_Q^{\oplus 2} \to \Ee \to \Ff \to 0,$$
then it is ACM, since the map
$$H^1(\Oo_Q(0,-2)\oplus\Oo_Q(-2,0))\rightarrow H^2(\Oo_Q(-2,-2)^{\oplus 2})$$
is an isomorphism. Now by the Chern class computation we have $\Ee\cong \Oo_Q(1,0)^{\oplus 2}\oplus \Oo_Q(0,1)^{\oplus 2}$. In particular, it is decomposable.

Assume that $\Ee$ is a general extension of $\Ff$ by $\Oo_Q$. Because of the generality, $\Ee$ cannot be $\Oo_Q(0,1)^{\oplus 2} \oplus \Oo_Q(2,0)$ nor $\Oo_Q(1,0)^{\oplus 2} \oplus \Oo_Q(0,2)$. Assume that $\Ee$ is decomposable. Since $(1,1)$ is not an index of $\Ee$, so $\Ee$ cannot be isomorphic to $\Oo_Q(1,1)\oplus \Aa_P$. Thus the only possibility is either $\Oo_Q(1,0)\oplus \Gg$ where $\Gg$ fits into the sequence (2) in Proposition \ref{prop1.2}, or $\Oo_Q(0,1)\oplus \Gg '$. In the former case, we have $h^0(\Gg(0,-1))=1$ and so $h^0(\Ee(0,-1))=1$, which is contradicting to $h^0(\Ee(0,-1))=2$. The latter case is also impossible similarly. Thus $\Ee$ is indecomposable.
\end{proof}

\begin{remark}
If  $\Ff$ is a non trivial extension of  $\Oo_Q(2,0)$ and $\Oo_Q(0,2)$ given in lemma \ref{nost}, we get $h^1(\Ff(-2,-2))=h^1(\Oo_Q(0,-2)\oplus \Oo_Q(-2,0))=2$
and $h^0(\Ff(0,-1))=h^0(\Oo_Q(2,-1)\oplus \Oo_Q(0,1))=2$. So the proof of the above lemma holds also in this case and we can conclude that
there exists an indecomposable extension $\Ee$ of $\Ff$ by trivial factors if and only if $\mathrm{rank}(\Ee)=3$.
\end{remark}

Now assume that $\Ff$ does not split with $h^1(\Ff^\vee)>0$. Firstly let us deal with the case when $\Ff$ is unstable.

\begin{lemma}
There exists no indecomposable vector bundle $\Ee$ of rank at least $3$, whose associated bundle $\Ff$ is non-splitting and unstable.
\end{lemma}
\begin{proof}
Since $\Ff$ has $h^1(\Ff^\vee )>0$, so it is one of the following:
\begin{enumerate}
\item $0\to \Oo_Q \to \Oo_Q(1,0)\oplus \Oo_Q(0,1)\oplus \Oo_Q(1,1) \to \Ff \to 0$
\item $0\to \Oo_Q(-1,-1)\to \Oo_Q^{\oplus 2} \oplus \Oo_Q(1,1) \to \Ff \to 0$.
\end{enumerate}

\quad{(1)} Since $h^1(\Ff^\vee)=1$, so the only non-trivial extension of $\Ff$ by a single $\Oo_Q$ cannot have a trivial factor. But it must be $\Oo_Q(1,0)\oplus \Oo_Q(0,1) \oplus \Oo_Q(1,1)$ from the sequence. In particular, it is decomposable.

\quad{(2)} Let us assume that the associated vector bundle of rank $2$ is from the sequence (2). We have $h^1(\Ff^\vee)=2$ and the vector bundle $\Ee$ has no trivial factor only if $\mathrm{rank}(\Ee)\le 4$. Note that $h^0(\Ee(-2,-2))=h^0(\Ff(-2,-2))=0$ and $h^1(\Ee(-3,-2))=h^1(\Ff(-3,-2))=0$. Similarly we have $h^1(\Ee(-2,-3))=0$. Using Theorem 6.1 in \cite{AO} with $t=-1$, $\Ee$ has $\Oo_Q(1,1)$ as its direct summand. In particular, $\Ee$ is decomposable.
\end{proof}

Now let us assume that $\Ff$ is stable and then we have $4\le c_2(\Ff) \le 8$.
\begin{remark}
If $c_2(\Ff)=8$, i.e. $\Ff$ is of maximal type with
$$0\to \Oo_Q(-2,-2) \to \Oo_Q^{\oplus 3} \to \Ff \to 0,$$
we have $h^1(\Ff^\vee) = 6$. Thus there exists an indecomposable globally generated vector bundle of maximal type for each rank $3\le r\le 8$ by Lemma \ref{mi}.
\end{remark}

The remaining cases are when $\Ff$ is from $\mathfrak{M}(c_2)$ with $4\le c_2 \le 6$.

\begin{proposition}
The extension of $\Ff \in \mathfrak{M}(4)$ by $\Oo_Q^{\oplus (r-2)}$ can be indecomposable if and only if $r=3$.
\end{proposition}
\begin{proof}
Let $\Ff \in \mathfrak{M}(4)$ have an index $(1,0)$. Then it fits into the exact sequence
$$0\to \Oo_Q(1,0) \to \Ff \to \Ii_Z(1,2) \to 0,$$
where $\deg (Z)=2$. Note that $h^1(\Ff^\vee)=h^1(\Ff(-2,-2))=h^1(\Ii_Z(-1,0))=2$. We have $h^1(\Ff ^\vee ) = h^1(\Ff (-2,-2)) = h^1(\Ii _Z(-1,0)) = \deg (Z) =2$.
Hence for each $r\ge 5$ and extension of $\Ff$ by $\Oo _Q^{\oplus (r-2)}$ has a trivial factor. Hence it is sufficient to check the cases
$r=3$ and $r=4$.
Let us take any extension $\Ee$ of $\Ff$ by $\Oo_Q^{\oplus 2}$
$$0 \to \Oo_Q^{\oplus 2} \to \Ee \to \Ff \to 0.$$
There is a non-zero map $\phi: \Oo _Q(1,0)^{\oplus 2} \to \Ee$. Set $\Gg:= \mathrm{Im}(\phi)$ and $c:= \mathrm{rank} (\phi )$.
First assume $c=1$. Since $\Gg$ is globally generated, torsion free, a quotient of $\Oo _Q(1,0)$ and $h^0(\Gg ^{\vee \vee })(-2,0)) \le h^0(\Ee (-2,0)) = 0$,
we get $\Gg= \Oo _Q(1,0)$; this is the case leading to $\Ww$, the indecomposable bundle of rank $3$ below, if it exists.

 Now assume $c=2$. Hence $\phi$ is injective and so $\Gg \cong \Oo _Q(1,0)^{\oplus 2}$. Let $\Cc$ be
the saturation of $\Gg$ in $\Ee$, i.e. the only rank two subsheaf of $\Ee$ containing $\Gg$ and then $\Dd:= \Ee/\Cc$ is a torsion free and globally generated sheaf of rank $2$. We have $c_1(\Cc )=(a,b)$
for some $a\ge 2$ and $b\ge 0$. Since $\Dd$ is globally generated, we get $a=2$ and $b\le 2$. First assume $b=0$. Since $\Gg \subseteq \Cc$ and
$\Gg$ and $\Cc$ are vector bundles with the same rank and the same determinant, then $\Cc = \Gg$. Since
$h^i(\Gg (0,-1)) = 0$, $i=0,1$, and $h^1(\Gg (0,-2))=0$ we get $h^0(\Dd (0,-1)) = h^0(\Ee (0,-1))=2$ and $h^0(\Dd (0,-2))=0$. We get $\Dd \cong \Oo_Q(0,1)^{\oplus 2}$
by the Chern class counting or $h^0(\Dd )=4$ and many other reasons.
Thus we have $\Ee \cong \Oo_Q(1,0)^{\oplus 2}\oplus \Oo_Q(0,1)^{\oplus 2}$. Now assume $b>0$. Note that $\Dd$ is globally generated , $c_1(\Dd)=(0,2-b)$ and $\mathrm{rank}(\Dd)=2$. Even if $\Dd$ is only torsion free, we have an exact sequence
$$0 \to \Oo _Q \to \Dd \to \Ii _A(0,1)\to 0$$
where $\Ii _A(0,1)$ is globally generated and thus we have $A=\emptyset$ and $\Dd \cong \Oo _Q\oplus \Oo_Q(0,1)$. Since $\Dd$ is a quotient of $\Ee$, Lemma \ref{trivial}
that $\Ee$ has a trivial factor, a contradiction. Hence there is no indecomposable extension of $\Ff$ by $\Oo_Q^{\oplus 2}$.

Now we prove the existence of the indecomposable extension of $\Ff$ by $\Oo_Q$.
Let us define $V: = H^0(\Oo _Q(1,0)^{\oplus 2}\oplus \Oo _Q(0,1)^{\oplus 2})$.
For a generally fixed $\phi : \Oo _Q\to \Oo _Q(1,0)^{\oplus 2}\oplus \Oo _Q(0,1)^{\oplus 2}$, let us set $\Ww := \mathrm{coker} (\phi)$. Since $\Oo _Q(1,0)^{\oplus 2}\oplus \Oo _Q(0,1)^{\oplus 2}$
is globally generated, $\Ww$ is a globally generated vector bundle of rank $3$ with $c_1=(2,2)$ and $c_2=4$. We claim that $\Ww$ is indecomposable. Let $\sigma : Q\to Q$
be the automorphism which exchanges the two rulings of $Q$. Since $\sigma ^\ast (\Oo _Q(1,0)^{\oplus 2}\oplus \Oo _Q(0,1)^{\oplus 2})
\cong \Oo _Q(1,0)^{\oplus 2}\oplus \Oo _Q(0,1)^{\oplus 2}$ and $\phi$ is general, the number of factors $\Oo _Q(1,0)$ of $\Ww$ is equal to the number
of factors $\Oo _Q(0,1)$ of $\Ww$. $\Oo _Q(1,1)$ is not a factor of $\Ww$, because $(1,1)$ is not an index of $\Ff$. Thus we have $\Ww \cong \Oo _Q(1,0)^{\oplus a}\oplus \Oo _Q(0,1)^{\oplus a}$ with $2a =3$, a contradiction.
\end{proof}

\begin{remark}
The proof gives a nice parameter space for the rank $3$ bundles, i.e. the projectivisation of the set of all nowhere vanishing sections of $H^0( \Oo_Q(1,0)^{\oplus 2}\oplus \Oo_Q(0,1)^{\oplus 2})$.
\end{remark}

Let $\mathfrak{G}_r$ with $r=3,4,5$ be the set of all vector bundles without trivial factors and which
are an extension of some $\Ff \in \mathfrak{M}(5)$ by $\Oo _Q^{\oplus (r-2)}$. Since
$h^1(\Ff ^\vee ) =3$ for all $\Ff \in \mathfrak {M}(5)$, each set $\mathfrak{G}_r$ is non-empty
and parametrized by an irreducible variety.

\begin{remark}\label{ss1}
Fix any $P\in \PP^3\setminus Q$ and look at the exact sequence
$$0 \to \Oo_Q(-1) \to \Oo _Q^{\oplus 3}\to \Aa _P\to 0$$
induced by the Euler sequence of $T\PP^2$. We get $h^1(\Aa _P^{\vee}(1,0)) =h^1(\Aa_P (0,-1))=0$.
Hence any extension of $\Aa _P$ by $\Oo _Q(1,0)$ is trivial. For the same reason any extension
of $\Aa _P$ by  $\Oo _Q(0,1)$ is trivial.
\end{remark}

\begin{proposition}\label{5.12} $ $
\begin{enumerate}
\item A general bundle $\Ee \in \mathfrak{G}_3$ is stable and thus indecomposable.
\item $\mathfrak{G}_4 = \{\Oo _Q(1,0)\oplus \Oo _Q(0,1)\oplus \Aa _P\}_{P\in \PP^3\setminus Q}$.
\item $\mathfrak{G}_5 = \{\Oo _Q(1,0)\oplus \Oo _Q(0,1)\oplus T\PP^3(-1)|_Q \}$. 
\end{enumerate}
\end{proposition}

\begin{proof}
\quad (a) Let us fix a general bundle $\Ee \in \mathfrak{G}_3$. We claim that $\Ee$ is indecomposable. Assume that this is not the case and so $\Ee$ has at least one line bundle, say $\Ll$, as a factor. Since $(1,1)$ is not an index of $\Ee$ and
$\Ee$ has no trivial factor, so we have either $\Ll\cong \Oo _Q(1,0)$ or $\Ll\cong \Oo _Q(0,1)$. Let us take $(a,b) \in \{(1,0),(0,1)\}$ such that
$\Ll \cong \Oo _Q(a,b)$. Since $\mathfrak{M}(5)$ is invariant by the automorphism of $Q$ which exchanges the two rulings of $Q$ (see Remark \ref{re44}),
$\mathfrak{G}_3$ is irreducible and $\Ee$ is irreducible, $\Oo _Q(b,a)$ is another factor of $\Ee$. Thus we should have $\Ee \cong \Oo_Q(1,0)\oplus \Oo_Q(0,1)
\oplus \Oo _Q(1,1)$, but this is impossible since $(1,1)$ is not an index of $\Ee$. Hence $\Ee$ is indecomposable.

Since $(1,1)$ is not an index of $\Ee$, every rank
$1$ subsheaf of $\Ee$ has $\Oo _Q(1,1)$-slope less than $4/3$. If $\Ee$ is not stable, we can take a rank $2$ saturated subsheaf $\Gg$ of $\Ee$ with maximal $\Oo _Q(1,1)$-slope $\alpha$
with $\alpha \ge 4/3$. Then the rank $1$ torsion free sheaf $\Ee/\Gg$ is globally generated and it has $\Oo _Q(1,1)$-slope at most $4/3$. Thus we have $(\Ee/\Gg )^{\vee \vee}
\cong \Oo _Q(c,d)$ with $(c,d)\in \{(0,0),(1,0),(0,1)\}$. Hence $\Ee/\Gg \cong \Ii _E(c,d)$ for some zero-dimensional scheme
$E\subset Q$. Since $\Oo _Q$ is not a direct factor of $\Ee$, we have $(c,d) \ne (0,0)$. We also have $E = \emptyset$ since $\Ii_E(c,d)$ is globally generated where $(c,d)\in \{(1,0),(0,1)\}$.  Hence $\Oo _Q(c,d)$ is a quotient of $\Ee$. Since $(c,d)$ is an index of $\Ee$, we easily get
that $\Oo_Q(c,d)$ is a factor of $\Ee$, a contradiction.

\quad (b) Fix a general $\Ee\in \mathfrak{G}_4$. Since $(1,0)$ (resp. $(0,1)$) is an index of $\Ee$, there is a non-zero map
$j_1: \Oo _Q(1,0) \to \Ee$ (resp. $j_2: \Oo _Q(0,1) \to \Ee$) and it is an injective map of sheaves. Consider the map $u =(j_1,j_2): \Oo_Q(1,0)\oplus \Oo _Q(0,1)
\to \Ee$. We first claim that $u$ has not rank $1$, i.e. that it is injective. Assume that $\mathrm{Im}(u)$ has rank $1$ and call $\Rr$ its saturation in $\Ee$.
Since $Q$ is a smooth surface, $\Rr$ is a line bundle, say $\Rr \cong \Oo _Q(a,b)$. Since neither $j_1$ nor $j_2$ is zero, we
have $a\ge 1$ and $b\ge 1$ and so $\Ee$ has $(1,1)$ as an index, a contradiction. Since $u$ is injective, the sheaf $\mathrm{Im}(u)$
is isomorphic to $\Oo _Q(1,0)\oplus \Oo _Q(0,1)$. We now check that $\mathrm{Im}(u)$ is saturated in $\Ee$, i.e. that $\Uu :=\mathrm{coker}(u)$ is torsion free.
Assume that $\mathrm{Im}(u)$ is not saturated in $\Ee$ and call $\Gg$ its saturation. Since $Q$ is a smooth surface, $\Gg$ is a vector bundle of rank $2$. Set $(a,b):= c_1(\Gg )$.
Since $\Ee$ is a general element of the irreducible family $\mathfrak{G}_4$ and the family $\mathfrak{G}_4$ is invariant under the automorphism of $Q$ which permutes
the two rulings of $Q$, we have $a=b$. Since $\mathrm{Im} (u) \nsubseteq \Gg$ and these two bundles have the same rank, we get $a>1$.
The rank $2$ torsion free sheaf $\Ee /\Gg$ is spanned and $c_1(\Ee /\Gg )(2-a,2-a)$. We get $a=2$ and that $\Ee /\Gg$ is trivial. Lemma \ref{trivial} gives
that $\Ee$ has a trivial factor, a contradiction. Hence $\mathrm{Im} (u)$ is saturated and so $\Uu$ is torsion free. Fix $P\in \PP^3 \setminus Q$.
For the bundle $\Oo _Q(1,0)\oplus \Oo _Q(0,1)\oplus \Aa _P$, the maps $j_1, j_2, u$ are uniquely defined and $\mathrm{coker} (u)$ is locally free. Thus $\Uu$ is locally free for a general $\Ee\in \mathfrak{G}_4$. The bundle $\Uu$ is a globally generated vector bundle of rank $2$ with $c_1(\Uu )=(1,1)$, no trivial factor and $h^0(\Uu )=h^0(\Ee )-4 =3$. We have $\Uu \cong \Aa _P$ for some $P\in \PP ^3\setminus Q$ by Proposition \ref{1.1}. By Remark \ref{ss1} a general $\Dd \in \mathfrak {G}_4$ is isomorphic to one of the bundles $\Oo_Q(1,0)\oplus \Oo_Q(0,1)\oplus \Aa _P$.

Let us fix any $\Ee \in \mathfrak {G}_4$ and then $\Ee$ is a degeneration of the family $\{\Oo _Q(1,0)\oplus \Oo_Q(0,1)\oplus \Aa _P\}_{P\in \PP^3\setminus Q}$. By semicontinuity we have $h^0(\Ee ^\vee (1,0) ) >0$ and $h^0(\Ee ^\vee (0,1) )>0$. Hence there is a map $f: \Ee \to \Oo _Q(1,0)\oplus \Oo _Q(1,0)$
with $\mathrm{Im}(f)$ of rank $2$. Since $\Ee$ has no trivial factor, $\mathrm{Im} (f)$ has no trivial factor by Lemma \ref{trivial}. Since $\mathrm{Im}(f)$ is globally generated, we get that $f$ is surjective and so $\ker (f)$ is a rank $2$ vector
bundle. Since each $\Aa _P$ is a quotient of $T\PP^3(-1)|_Q$, the semicontinuity theorem gives $h^0(T\PP^3(-1)|_Q \otimes \Ee ^\vee )>0$.
Let us fix a non-zero map $h: T\PP^3(-1)|_Q \to \Ee$.
Since $T\PP^3(-1)|_Q$ has neither $(1,0)$ nor $(0,1)$ as an index, we have $f\circ h = 0$, i.e. $\mathrm{Im} (h) \subseteq \ker (f)$. Since $T\PP^3(-1)|_Q$ is stable and neither $(2,0)$ nor $(1,1)$ is
an index of $\Ee$, $\mathrm{Im} (h)$ has rank $2$. Since $\mathrm{Im} (h)$ is globally generated, we get that $\mathrm{Im} (h)$ is a quotient of $\Oo _Q^{\oplus 3}$ by a map $\Oo _Q(-1,-1)\to \Oo _Q^{\oplus 3}$ defined by
$3$ linearly independent forms vanishing at some point $O\in \PP^3$. We also see that $\ker (f)$ is the saturation of $\mathrm{Im} (h)$ in $\Ee$. If $O\notin Q$, then
$\mathrm{Im}(h)$ is saturated in $\Ee$ and so $\mathrm{Im} (h) =\ker (f) \cong \Aa _O$. It implies that $\Ee \cong \Oo _Q(1,0)\oplus \Oo _Q(0,1)\oplus \Aa _O$. Now assume
$O\in Q$ and then $\mathrm{Im} (h)$ is not locally free. In particular, we have $\mathrm{Im} (h)\ne \ker (f)$. Note that the sheaves $\mathrm{Im} (h)$ and $\ker (f)$ have the same Chern numbers, while $\ker (f)/\mathrm{Im} (h)$ is supported
by a single point $O$. Thus we have $c_2(\mathrm{Im} (h)) > c_2(\ker (f))$, a contradiction.

\quad (c) Let us take any $\Ee \in \mathfrak{G}_5$. Since $\Ee$ is an extension of some $\Gg\in \mathfrak{G}_4$ by $\Oo _Q$, it is sufficient to use
the rank $4$ case.
\end{proof}

%\begin{proposition}
%There exists an indecomposable and globally generated vector bundle of rank $r\ge 2$ on $Q$ with $c_1=(2,2)$ and $c_2=5$ if and only if $2\le r \le 5$.
%\end{proposition}
%\begin{proof}
%$\Ff\in \mathfrak{M}(5)$ fits into the sequence
%$$0\to \Oo_Q(1,0) \to \Ff \to \Ii_Z(1,2) \to 0$$
%with $\deg (Z)=3$ and so we have $h^1(\Ff^\vee)=3$. So there exists an indecomposable extension $\Ee$ of $\Ff$ by trivial factors only if $\mathrm{rank}(\Ee)\le 5$. Let $\Gg$ be a globally generated vector bundle of rank $2$ with $c_1=(1,2)$ and $c_2=3$ and so it fits into
%$$0\to \Oo_Q(-1,-1) \to \Oo_Q^{\oplus 2} \oplus \Oo_Q(0,1) \to \Gg \to 0.$$
%In particular, we have $h^1(\Gg^\vee (1,0))=h^1(\Gg(0,-2))=2$ and so there exists a non-trivial extension $\Ee$ of $\Gg$ by $\Oo_Q(1,0)$. Since $h^0(\Ee^\vee)=0$ and $h^0(\Ee(-1,0))=h^1(\Ff(-1,0))=1$, so $\Ee$ must be an indecomposable vector bundle of rank $3$ with given conditions in the assertion.
%\end{proof}

Now the only remaining case is when $c_2(\Ff)=6$ and this case can be divided into two parts depending on the index.

\begin{proposition}
There exists an indecomposable and globally generated vector bundle of rank $r\ge 3$ on $Q$ with $c_1=(2,2)$, $c_2=6$ with index $(1,0)$ or $(0,1)$ if and only if we have $r\le 5$.
\end{proposition}
\begin{proof}
Without loss of generality, let us assume that the index of $\Ff$ is $(1,0)$, i.e $\Ff$ fits into the sequence
$$0\to \Oo_Q(-1,-2) \to \Oo_Q^{\oplus 2} \oplus \Oo_Q(1,0) \to \Ff \to 0$$
and so $h^1(\Ff^\vee)=4$. Thus we have $\mathrm{rank}(\Ee)\le 6$. Assume that $\Ee$ is decomposable. If $\Ee$ does not have a line bundle as its direct factor, then we have $\Ee \cong \Ee_1 \oplus \Ee_2$ with $c_1(\Ee_1)=c_1(\Ee_2)=(1,1)$. From the classification of indecomposable and globally generated vector bundles of rank at least $2$ with $c_1=(1,1)$, the only options for $\Ee_1$ and $\Ee_2$ are either $\Aa_P$ or $T\PP^3(-1)|_Q$ since $\Ee$ does not have any trivial factor. But it is impossible due to the second Chern class counting and the index. Thus $\Ee$ should have at least one line bundle as its direct factor. Since $h^0(\Ff(0,-1))=0$, the direct factor must be $\Oo_Q(1,0)$ and so we have $\Ee \cong \Oo_Q(1,0)\oplus \Gg$ for a globally generated vector bundle with $c_1(\Gg)=(1,2)$ and $c_2(\Gg)=4$. In particular, $\Gg$ is of maximal type with
$$0\to \Oo_Q(-1,-2) \to \Oo_Q^{\oplus  r}\to \Gg \to 0.$$
From the long exact sequence
$$0\to H^0(\Oo_Q(1,0)^{\oplus r}) \to H^0 (\Oo_Q(2,2)) \to H^1(\Gg^\vee (1,0)),$$
we have $h^1(\Gg^\vee (1,0))>0$ for $r\le 4$. In other words, there exists a non-trivial extension $\Ee'$ of $\Gg$ by $\Oo_Q(1,0)$. Since $h^0(\Ee(-1,0))=h^0(\Ff(-1,0))=1$, so $\Ee'$ cannot be isomorphic to $\Oo_Q(1,0)\oplus \Gg '$ and thus $\Ee'$ is indecomposable. Since $h^1(\Oo_Q(1,0))=0$ and $\Oo_Q(1,0)$ is globally generated, so $\Ee'$ is globally generated. Hence there exists an indecomposable and globally generated vector bundle of rank up to $4$ with $c_1=(2,2)$, $c_2=6$ and index $(1,0)$.

Now let us consider the case of rank $5$. Let us define $\mathbb {G}_{r-1}$ with $3 \le r \le 6$ to be the set of all $\Gg _{r-1}$ fitting into an exact sequence
\begin{equation}\label{eqa1}
0 \to \Oo _Q(-1,-2) \to \Oo _Q^{\oplus r} \to \Gg _{r-1}\to 0
\end{equation}
and with no trivial factor, i.e. $\Gg_{r-1}$ is a globally generated vector bundle of rank $r$ with no trivial factor, $c_1 =(1,2)$ and maximal type.
Let us fix a bundle $\Gg _5$ and set $V_1:= H^0(\Gg _5)$, $V_2:= H^0(\Oo _Q(1,0))$ and $V:= V_1\oplus V_2$.
For a generally fixed $\phi \in V$, say $\phi  = (\phi _1,\phi _2)$, let us set $\Ee := \mathrm{coker} (\phi)$. It is sufficient
to prove that $\Ee$ is indecomposable. Assume that this is not the case. We saw above that $\Ee  \cong \Gg _4\oplus \Oo _Q(1,0)$
for some $\Gg _4\in \mathbb {G}_4$. Hence $\phi$ induces an exact sequence
$$0 \to \Oo _Q \stackrel{\phi}{\to} \Gg _5\oplus \Oo _Q(1,0)\stackrel{\psi}{\to}\Gg _4\oplus \Oo _Q(1,0)\to 0$$
with $\psi$ represented as a $2\times 2$ matrix
\begin{displaymath}
\left(\begin{array}{ccc}
\psi _{11} & \psi _{12} \\
\psi _{21} & \psi _{22}
\end{array} \right)
\end{displaymath}
with $\psi _{11} : \Gg _5 \to \Gg _4$, $\psi _{12}: \Oo _Q(1,0) \to \Gg _4$, $\psi _{21}: \Gg _5 \to \Oo _Q(1,0)$ and $\psi _{22}: \Oo _Q(1,0)\to \Oo _Q(1,0)$.
From (\ref{eqa1}) we get $H^0(\Gg _4(-1,0)) =0$. Hence $\psi _{12}=0$. Dualizing (\ref{eqa1}) we get $H^0({\Gg_5}^\vee (1,0)) =0$,
because the map $\Oo _Q(-1,-2) \to  \Oo _Q^{\oplus 6}$ is induced by the complete linear system $\vert \Oo _Q(1,2)\vert$.
Hence we have $\psi _{21}=0$, and so $\psi _{11}$ is surjective and $\psi _{22}$ is an isomorphism. Thus we have $\ker (\psi ) \subset \Gg _5\oplus \{0\}$ and so $\mathrm{Im} (\phi ) \subset \Gg _5\oplus \{0\}$.
Since $H^0({\Gg _5}^\vee (1,0)) =0$, there
is a unique subsheaf of  $\Gg _5\oplus \Oo _Q(1,0)$ isomorphic to $\Gg _5$. Since $\Oo _Q(1,0)$ is globally generated and $\phi$ is general, we get a contradiction.

By Proposition \ref{split}, we do not have an indecomposable extension of $\Ff$ by $\Oo_Q^{\oplus 4}$.
\end{proof}

\begin{proposition}\label{split}
Every globally generated vector bundle of rank $6$ with $c_1=(2,2)$, $c_2=6$, index $(1,0)$ and without trivial factors, has $\Oo_Q(1,0)$ as its direct factor.
\end{proposition}
\begin{proof}
 We will prove that any such a bundle is isomorphic to $\Gg _5\oplus \Oo _Q(1,0)$. Since the set of all $\Ff$ with given numeric data is irreducible, the set of all such bundles $\Ee$ of rank $6$ is parametrized by an irreducible variety $\Gamma $.

\quad {\emph {Claim 1:}} If $\Cc$ is a proper subsheaf of $\Gg _5$ generated by its global sections, then $\Cc \cong \Oo _Q^{\oplus c}$ for some $c$.

\quad {\emph {Proof of Claim 1:}} Set $c:= \mathrm{rank} (\Cc)$. Since $\Cc$ is globally generated, we have $\Cc \cong \Oo _Q^{\oplus c}$
if and only if $h^0(\Cc) \le c$. If $c=5$, then we have $\Cc \cong \Oo _Q^{\oplus 5}$ because $h^0(\Gg _5) = 6$ and $h^0(\Cc) < h^0(\Gg _5)$ which is due to the global generatedness of $\Gg _r$. Now assume that $1\le c \le 4$ and $d:= h^0(\Cc) >c$. Let $\Cc'$ be the saturation of $\Cc$ in $\Gg _r$. The torsion free sheaf $\Gg _5/\Cc'$ is generated by $H^0(\Gg _r)/H^0(\Cc)$,
i.e. by a vector space of dimension at most $5-c =\mathrm{rank} (\Gg _5/\Cc')$. Thus the sheaf $\Gg _5/\Cc'$ is trivial and so $\Gg _5$ has a trivial factor, a contradiction.

Notice that $\Gg _5$ is unique up to isomorphisms because it is induced by the complete linear system $\vert \Oo _Q(1,2)\vert$ and so we have $u^\ast (\Gg _5) \cong \Gg _5$ for each $u\in \mathrm{Aut} (\PP^1 )\times \mathrm{Aut} (\PP^1 )$. Hence it is uniform with respect to both
system of lines. Its splitting type is $(1,0,0,0,0,0)$ with respect to one of the system of lines and $(1,1,0,0,0,0)$ with respect to the other one, because the cokernel of a general map $\Oo _{\PP^1}(-2) \to \Oo _{\PP^1}^{\oplus 3}$ has splitting type $(1,1)$ and thus its splitting type is not $(2,0,0,0,0,0)$. Hence the bundle $\Gg _5\oplus \Oo _Q(1,0)$ is uniform of type $(1,1,0,0,0,0)$ with respect to both systems of lines. Let $\Ee$ be a general element of $\Gamma$. 

\quad{\emph {Claim 2:}} $\Ee$ is uniform of splitting type $(1,1,0,0,0,0)$, i.e. for each line $D\subset Q$, the bundle $\Ee|_ D$ is a direct sum of two degree $1$ line bundles and four degree $0$ line bundles.

\quad{\emph{Proof of Claim 2.}} Let us define a set
$$\tilde{\Gamma}:= \{(\Ee ,D) \in \Gamma \times \vert \Oo _Q(1,0)\vert : \Ee |_ D \text{ has splitting type } (2,0,0,0,0,0)\}.$$
 Here we see $\Gamma$ as an irreducible algebraic variety parametrizing all bundles in $\Gamma$ (not necessarily one-to-one). By semicontinuity $\tilde{\Gamma}$ is closed in $\Gamma \times \vert \Oo _Q(1,0)\vert $. Let $\pi _1: \tilde{\Gamma} \to \Gamma$ be the restriction of the first projection $p_1: \Gamma \times \vert \Oo _Q(1,0)\vert  \to \Gamma$ to $\tilde{\Gamma}$. Since $p_1$ is proper and $\tilde{\Gamma}$ is closed in $\Gamma \times \vert \Oo _Q(1,0)\vert $, the map $\pi _1$ is proper. The bundle  $\Gg_5 \oplus \Oo_Q(1,0)$, is uniform with splitting type $(1,1,0,0,0,0)$ and thus $\pi _1 (\tilde{\Gamma})$ is a proper closed subset of the variety $\Gamma$. Hence $\Ee \vert_ D$ has splitting type $(1,1,0,0,0,0)$ for every $D\in \vert \Oo_Q(0,1)\vert$. The same proof works for any $D\in \vert \Oo_Q(0,1)\vert$.

Let us fix a uniform bundle $\Ee \in \Gamma$. Since $h^0(\Ee (-1,0)) =1$, there is an injective map $u: \Oo _Q(1,0) \to \Ee$. Let us set $\Mm = \mathrm{coke}r (u)$. Since $h^0(\Ee (-2,0)) = h^0(\Ee (-1,-1)) =0$, so $\mathrm{Im} (u)$ is saturated in $\Ee$ and thus $\Mm$ is torsion free, i.e. $u$ is an embedding of vector bundles outside finitely many points of $Q$.
We want to prove that $\Mm$ is locally free. Since $\Ee$ is uniform of splitting type $(1,1,0,0,0,0)$ with respect to the ruling $\vert \Oo _Q(0,1)\vert$, while
$\Oo _Q(1,0)$ has splitting type $1$ with respect to the same ruling, $u\vert _D$ is an embedding of vector bundles
for each $D\in \vert \Oo _Q(0,1)\vert$. It implies that $\Mm$ is locally free. Since $\Mm$ is a globally generated vector bundle with $c_1=(1,2)$, no trivial factor and $h^0(\Mm (-1,0))
= h^0(\Mm (0,-1)) =0$, we have $\Mm \cong \Gg _5$. Look
again at the exact sequence
$$H^0(\Oo_Q(1,0))^{\oplus 6} \stackrel{v}{\to} H^0 (\Oo_Q(2,2)) \to H^1({\Gg_5} ^\vee (1,0)) \to 0.$$
Any extension of $\Gg _5$ by $\Oo _Q(1,0)$ splits if and only if the map $v$ is surjective. Since $\Oo _Q(1,0)$ is globally generated and $h^0(\Oo _Q(1,0)) =2$, we have an exact sequence
\begin{equation}\label{eqaa1}
0 \to \Oo _Q(-1,0) \to H^0(\Oo _Q(1,0))\otimes \Oo _Q \to \Oo _Q(1,0) \to 0.
\end{equation}
Since $h^1(\Oo _Q(0,2))=0$, by tensoring the sequence (\ref{eqaa1}) with $\Oo _Q(1,2)$ we get the surjectivity
of the multiplication map
$$ H^0(\Oo _Q(1,0))\otimes H^0(\Oo _Q(1,2)) \to H^0(\Oo _Q(2,2)),$$ which is the map $v$. In other words, $v$ is induced by the complete linear system $|\Oo_Q(1,2)|$ and so we have $h^1(\Gg_5^\vee(1,0))=0$.

 In summary, we just proved the
existence of a non-empty open subset $U$ of $\Gamma$ such that $\Ee _\gamma \cong \Oo _Q(1,0) \oplus \Gg _5$ for all $\gamma \in U$
and that $U$ contains all uniform vector bundles in $\Gamma$ (and all $\Ee \in \Gamma$ for which the unique inclusion $\Oo _Q(1,0) \to \Ee$ has a locally free
quotient). By semicontinuity we have $h^0(\Ee ^\vee \otimes \Gg _5 )>0$. For a fixed non-zero $w: \Ee \to \Gg _5$, $\mathrm{Im}(w)$ is globally generated since $\Ee$ is globally generated.
Since $\Ee$ has no trivial factor, so $\mathrm{Im} (w)$ has no trivial factor. {\it Claim 1} gives that $w$ is surjective and thus $\ker (w) $ is a line bundle. Hence $\Ee$ is
an extension of $\Gg _5$ by $\Oo _Q(1,0)$ and thus we have $\Ee \cong \Gg_5 \oplus \Oo_Q(1,0)$ since $h^1(\Gg_5^\vee (1,0))=0$.
\end{proof}

As the final case let us consider when $\Ff$ is globally generated vector bundle of rank $2$ with $c_1=(2,2)$, $c_2=6$ and index $(0,0)$. In particular, we have $h^1(\Ff^\vee)=4$ and thus there exists an indecomposable extension $\Ee$ of $\Ff$ by trivial factors only if $\mathrm{rank}(\Ee)\le 6$.

\begin{lemma}\label{oo}
For two points $P, O\in \PP^3\setminus Q$,  we have $\Aa _P \cong \Aa _O$ if and only if $P=O$.
\end{lemma}
\begin{proof}
The assertion follows from the fact that the set of jumping conics of $\Aa_P$ is the hyperplane in $(\PP^3)^*$ corresponding to $P$ \cite{huh2}.
\end{proof}

\begin{lemma}\label{rm1}
For two points $P$ and $O$ in $\PP^3\setminus Q$, we have :
$$h^1(\Aa_O^\vee \otimes \Aa_P)=\left\{
                                                         \begin{array}{ll}
                                                         3, & \hbox{if $\Aa_O\cong \Aa_P$}; \\
                                                         2, & \hbox{if $\Aa_O \not\cong \Aa_P$}.
                                                         \end{array}
                                                      \right.$$
\end{lemma}

\begin{proof}
Since each $\Aa$ is stable, we have $h^0(\Aa _O^\vee \otimes \Aa _P)=0$ if $P\ne O$ and $h^0(\Aa _O^\vee \otimes \Aa _P)=1$
if $P=O$ by Lemma \ref{oo}. Duality
gives $h^2(\Aa _O^\vee \otimes \Aa _P) = h^0(\Aa _O\otimes \Aa _P^\vee (-2,-2)) =0$ and so we can apply Hirzebruch Riemann Roch to obtain the assertion.

Alternatively let us consider the  exact sequence
\begin{equation}\label{eqrm1}
0 \to \mathcal {A}_O^\vee (-1) \to (\mathcal {A}_O^\vee )^{\oplus 3} \to \mathcal {A}_O^\vee \otimes \mathcal{A}_{P} \to 0
\end{equation}
and then we have
\begin{align*}
&h^2(\Aa_O^\vee (-1)) = h^1(\Aa _O(-2)) = h^1(T\PP ^2(-3)) + h^1(T\PP ^2(-4)) =1,\\
&h^1(\mathcal {A}_O^\vee(-1) )=h^1(\Omega _{\PP^3}) +h^1(\Omega _{\PP^3}(-1)) =1, \text{ and}\\
&h^1(\mathcal {A}_O^\vee )=h^1(\Omega _{\PP^2}(1)) +h^1(\Omega _{\PP^2}) =1.
\end{align*}
Then we can use (\ref{eqrm1}) to obtain the assertion.
\end{proof}

\begin{proposition}\label{o1}
Let $\Ee$ be a vector bundle fitting into a non-splitting exact sequence
\begin{equation}\label{eqo1}
0 \to \Aa _O \stackrel{j}{\to} \Ee \stackrel{f}{\to} \Aa _P \to 0.
\end{equation}
Then $\Ee$ is semistable and simple. Moreover, $\Aa _O$ is the only subsheaf of $\Ee$ with slope $1/2$.
\end{proposition}

\begin{proof}
Since $c_1(\Aa )=1$ and $h^0(\Aa (-1,-1)) =0$, so $\Aa$ is stable. Since $h^0(\Ee (-1,-1)) =0$, there
is no non-zero map $\Ll \to \Ee$ with $\Ll$ a line bundle and $\deg (\Ll)>0$. Since $h^0(\Ee ^\vee )=0$, there
is no non-zero map $\Ee \to \Rr$ with $\Rr$ a line bundle and $\deg (\Rr)\le 0$. Hence there is no injective
map $\Ff \to \Ee$ with a sheaf $\Ff$ of rank $3$ whose degree is at least $2$, i.e. the slope of $\Ff$ is at least $1/2$.
Let $\Gg$ be a rank $2$ subsheaf of $\Ee$ different from $j(\Aa _O)$, saturated in $\Ee$ and with
maximal slope, say $a$. Since $\Ee/\Ff$ has no torsion and $Q$ is a smooth surface, $\Ff$ is locally free. Assume $a\ge 1/2$. Since $h^0(\Ff (-1,-1)) \le h^0(\Ee (-1,-1))=0$, $\Ff$ is stable.
Hence $f\vert_{ \Ff} : \Ff \to \Aa_P$ is either zero or an isomorphism. In the former case we have $\Ff \subseteq \Aa_O$, a contradiction. In the latter case
$f\vert_{ \Ff}$ induces a splitting of (\ref{eqo1}), a contradiction.

Let us assume that $\Ee $ is not simple, i.e. assume the existence
of a non-zero map $h: \Ee \to \Ee$ with $m:= \mathrm{rank}(h(\Ee )) <4$. Set $d:= \deg (\mathrm{Im} (h))$. Since $\Ee$ is semistable and $\mathrm{Im}(h)$ is both
a quotient of $\Ee$ and a subsheaf of $\Ee$, we get $d/m = 1/2$, i.e. $m=2$ and $d=1$. We just saw that that $j(\Aa _O)$ is the saturation of $\mathrm{Im} (h)$ and so $h$ gives a splitting of (\ref{eqo1}), a contradiction.
\end{proof}

\begin{remark}\label{re17}
We can similarly obtain $h^1(T\PP^3(-1)|_Q \otimes T\PP^3(-1)|_Q^\vee)=0$ and thus there is no non-trivial extension of $T\PP^3(-1)|_Q$ by itself.
The exact sequence
$$0 \to T\PP^3(-1)|_Q ^\vee(1,1) \to \Oo _Q(1,1)^{\oplus 4} \to \Oo _Q(2,2) \to 0$$
gives $h^0(T\PP^3(-1)|_Q ^\vee (1,1)) =0$.
Since $h^0(T\PP^3(-1)|_Q (-1,-1)) =0$, so $T\PP^3(-1)|_Q$ is $\Oo _Q(1,1)$-stable.
\end{remark}

\begin{lemma}
For any $P\in \PP^3\setminus Q$ we have
\begin{enumerate}
\item $h^0(\Aa _P\otimes T\PP^3(-1)|_Q^\vee)=1$ and $h^1(\Aa _P \otimes T\PP^3(-1)|_Q^\vee)=0$;
\item $h^0({\Aa _P}^\vee \otimes T\PP^3(-1)|_Q)=0$ and $h^1(\Aa_P^\vee \otimes T\PP^3(-1)|_Q))=4$.
\end{enumerate}
In particular, there is no non-trivial extension of $T\PP^3(-1)|_{Q})$ by $\Aa _P$ and there exist non-trivial extensions of $\Aa _P$ by $T\PP^3(-1)|_{Q})$.
\end{lemma}

\begin{proof}
For the first assertion, let us consider the sequence
$$ 0\to T\PP^3(-1)|_Q^\vee(-1,-1) \to (T\PP^3(-1)|_Q^\vee)^{\oplus 3} \to \Aa_P \otimes T\PP^3(-1)|_Q^\vee \to 0$$
to obtain
\begin{align*}
0\to H^0(\Aa_P \otimes T\PP^3(-1)|_Q^\vee) &\to H^1( T\PP^3(-1)|_Q^\vee (-1,-1)) \to H^1(T\PP^3(-1)|_Q^\vee)^{\oplus 3} \\
&\to H^1(\Aa_P \otimes T\PP^3(-1)|_Q^\vee) \to H^2(T\PP^3(-1)|_Q^\vee (-1,-1)).
\end{align*}
Since $h^1(T\PP^3(-1)|_Q^\vee (-1,-1))=h^1(T\PP^3(-1)|_Q(-1,-1))=h^2(\Oo_Q(-2,-2))=1$, $h^1(T\PP^3(-1)|_Q^\vee)=0$ and $h^2(T\PP^3(-1)|_Q^\vee (-1,-1))=0$, so we have the assertion.

For the second assertion, we have $h^0(\Aa_P^\vee \otimes T\PP^3(-1)|_Q)=0$ since both of $\Aa_P$ and $T\PP^3(-1)|_Q$ are stable. Thus we have $h^1(\Aa_P^\vee \otimes T\PP^3(-1)|_Q)=4$ similarly.
\end{proof}

\begin{proposition}\label{o2}
Let $\Ee$ be a vector bundle fitting into a non-splitting exact sequence
\begin{equation}\label{eqo2}
0 \to T\PP^3(-1)|_Q\stackrel{j}{\to} \Ee \stackrel{f}{\to} \Aa _P \to 0.
\end{equation}
Then $\Ee$ is stable.
\end{proposition}

\begin{proof}
Let us assume that $\Ee$ is not stable and take a proper saturated subsheaf $\Ff$ of $\Ee$ with maximal slope. Set $a:= \deg (\Ff)$ and $m:= \mathrm{rank} (\Ff)$.
Since $\Ff$ is saturated
in $\Ee$ and $Q$ is a smooth surface, $\Ff$ is locally free. Since $\Ee$ has slope $1/2$, we get $a/m \ge 1/2$. From (\ref{eqo2}) we get $h^0(\Ee (-1)) = h^0(\Ee ^\vee (1)) =0$. Thus $m$ is neither $1$ nor $3$, and so we have $(m,a)=(2,1)$. Since $\Aa _P$ is stable and of slope $1/2$, so the map $h:=f\vert_{ \Ff}$ is either a zero map or an isomorphism. If $h$ is an isomorphism, then (\ref{eqo2}) splits. If $h=0$, then $\Ff \subset T\PP^3(-1)|_Q$, contradicting the stability of $T\PP^3(-1)|_Q$.
\end{proof}

\begin{proposition}
There exists an indecomposable and globally generated vector bundle of rank $r\ge 2$ on $Q$ with $c_1=(2,2)$, $c_2=6$ and index $(0,0)$ if and only if $2\le r \le5$.
\end{proposition}
\begin{proof}
From the previous lemmas, we have the existence of indecomposable vector bundles $\Ee$ with the conditions in the assertion if $\mathrm{rank}(\Ee)\le 5$. When $\mathrm{rank}(\Ee)=6$, we claim that every such a bundle is isomorphic to $(T\PP^3(-1)|_Q)^{\oplus 2}$ and it follows from Proposition \ref{uu1}.
\end{proof}

Let $\mathfrak{U}$ be the set of all globally generated vector bundles of rank $6$ with $c_1=(2,2)$, $c_2=6$ and index $(0,0)$. Since the set of all rank $2$ bundles $\Ff$ with the same numeric data
is parametrized by an irreducible variety $U'$ \cite{Soberon}, so the set $\mathfrak {U}$ is parametrized by an irreducible variety, say $U$, because $U$ is an open subset of a quotient bundle over $U'$ from sequence (\ref{heqa}).

\begin{proposition}\label{uu1}
We have $\mathfrak {U} = \{(T\PP^3(-1)|_Q)^{\oplus 2}\}$.
\end{proposition}
\begin{proof}
Obviously $(T\PP^3(-1)|_Q)^{\oplus 2} \in \mathfrak {U}$. Since $h^1(T\PP^3(-1)|_Q ^\vee \otimes T\PP^3(-1)|_Q )=0$ from Remark \ref{re17}, we have $h^1(\mathcal{E}nd (T\PP^3(-1)|_Q \oplus T\PP^3(-1)|_Q ))=0$ and so the vector bundle $(T\PP^3(-1)|_Q)^{\oplus 2}$ is rigid. Thus there is a non-empty open subset $V$ of $U$ such
that $\Ee \cong (T\PP^3(-1)|_Q)^{\oplus 2}$ for
all $\Ee \in V$.

Let us fix $\Ee \in U$. Since $h^0(\mathcal{H}om (T\PP^3(-1)|_Q ,\Cc)) = h^0(\mathcal{H}om (\Cc ,T\PP^3(-1)|_Q ))  =2$ for all $\Cc \in V$,
the semicontinuity theorem for cohomology gives us the inequalities $h^0(\mathcal{H}om (T\PP^3(-1)|_Q ,\Ee ))\ge 2$ and $h^0(\mathcal{H}om (\Ee ,T\PP^3(-1)|_Q )) \ge 2$.
Assume the existence of $h: \Ee \to T\PP^3(-1)|_Q$ such that $\mathrm{Im} (h)$ has rank $1$ and write $\Oo _Q(a,b):= \mathrm{Im} (h)^{\vee \vee }$. Since $T\PP^3(-1)|_Q$ is locally free, it contains
$\mathrm{Im} (h)^{\vee \vee }$. Since $(0,1)$ and $(1,0)$ are not indices of $T\PP^3(-1)|_Q$, we get $a=b=0$. Since $\mathrm{Im} (h)$ is globally generated, we get $\mathrm{Im} (h) \cong \Oo _Q$. Hence
$\Ee$ has a trivial factor, a contradiction. Now assume the existence of $h : \Ee \to T\PP^3(-1)|_Q$ such that $\mathrm{Im} (h)$ has rank $2$. Set $\Oo _Q(u,v):= c_1(\mathrm{Im} (h)^{\vee \vee })$.
Since $\Ee$ is globally generated, we have $u\ge 0$ and $v\ge 0$. Since $T\PP^3(-1)|_Q$ is stable, we have $u+v \le 1$. Since $\mathrm{Im} (h)$ is a torsion free and globally generated sheaf, it fits
into an exact sequence
$$0 \to \Oo _Q \to \mathrm{Im} (h) \to \Ii _Z(u,v)\to 0$$
for some zero-dimensional scheme $Z$. Since $\Ii _Z(u,v)$ is globally generated and $(u,v) \in \{(0,0),(1,0),(0,1)\}$, we get $Z=\emptyset$. $\mathrm{Im} (u) \cong \Oo _Q\oplus \Oo _Q(u,v)$.
In particular, $\Oo _Q$ is a factor of $\Ee$ and so every non-zero map $h: \Ee \to T\PP^3(-1)|_Q$ has rank $3$. In the same way we see that any non-zero map
$T\PP^3(-1)|_Q \to \Ee$ has image of rank $3$. Recall that $H^0(\mathcal{H}om (\Ee ,T\PP^3(-1)|_Q ))$ is a vector space of dimension $\ge 2$. Fix a non-zero map $f : T\PP^3(-1)|_Q \to \Ee$. 

First assume $h\circ f \ne 0$.
Since $T\PP^3(-1)|_Q$ is stable, so $h\circ f : T\PP^3(-1)|_Q \to T\PP^3(-1)|_Q$ is an isomorphism. Thus $T\PP^3(-1)|_Q$ is a factor
of $\Ee$ and so we have $\Ee \cong T\PP^3(-1)|_Q \oplus \Vv$ for some globally generated vector
bundle $\Vv$ of rank $3$ with $c_1=(1,1)$, $h^0(\Vv )=4$ and no trivial factor. Hence we have $\Ee \cong T\PP^3(-1)|_Q \oplus T\PP^3(-1)|_Q$.

Now assume $h\circ f =0$, i.e. $\mathrm{Im} (f) \subseteq \ker (h)$. Since $f$ is injective, $\mathrm{Im} (f)\cong T\PP^3(-1)|_Q$. The sheaf $\ker (h)$ is reflexive, because it is a kernel of a map between two
vector bundles. Since $Q$ is a smooth surface, $\ker (h)$ is locally free. First assume  $c_1(\mathrm{Im} (h)^{\vee \vee}) =(0,0)$. Since
$\Ee$ is spanned, so is $\mathrm{Im} (u)$. Our assumption on $c_1$ implies that $\mathrm{Im} (h) \cong \Oo _Q^{\oplus 3}$. Since $\Ee$ is spanned, we get
that $\Oo _Q^{\oplus 3}$ is a factor of $\Ee$, a contradiction. Now assume that either $c_1(\mathrm{Im} (h)^{\vee \vee}) =(1,0)$ or $c_1(\mathrm{Im} (h)^{\vee \vee}) =(0,1)$. Since $\mathrm{Im} (h)$ is a rank 3 globally
generated sheaf, but not a trivial vector bundle, we have $h^0(\mathrm{Im} (h))\ge 4$. Since $h^0(T\PP^3(-1)|_Q )=4$ and $T\PP^3(-1)|_Q$ is spanned, we get $\mathrm{Im} (h) =T\PP^3(-1)|_Q$,
a contradiction.

Now assume $c_1(\mathrm{Im} (h)^{\vee \vee}) =(1,1)$, i.e. $c_1(\ker (h))=(1,1)$.
The inclusion $j: \mathrm{Im} (f) \subseteq \ker (h)$ is a injective map between two vector bundles with the same rank and the same
determinant. Hence $j$ is surjective. Recall that $h^0(\mathcal{H}om (T\PP^3(-1)|_Q ,\Ee ))\ge 2$. Take $f_1\in H^0(\mathcal{H}om (T\PP^3(-1)|_Q ,\Ee ))$ linearly independent
from $f$. We get $\mathrm{Im} (f_1) \subseteq \ker (h)$. Since $\ker (h) \cong T\PP^3(-1)|_Q$ and $T\PP^3(-1)|_Q$ is stable, we get a contradiction.
\end{proof}

%%%%%%%%%%%%%%%%%%%%%%%%%%%%%%%%%%%%%%%%%%
\providecommand{\bysame}{\leavevmode\hbox to3em{\hrulefill}\thinspace}
\providecommand{\MR}{\relax\ifhmode\unskip\space\fi MR }
\providecommand{\MRhref}[2]{%
  \href{http://www.ams.org/mathscinet-getitem?mr=#1}{#2}
}
\providecommand{\href}[2]{#2}

\end{document}